\documentclass[a4paper,11pt]{article}

\pdfoutput=1
\usepackage{microtype}
\usepackage[subsec]{modstyle}
\usepackage{bm}
\newcommand{\CRig}{\catname{CRig}}

\newcommand{\pvto}[2]{{\kern1pt\left(\begin{matrix}\text{\small $#1$}\\[-5pt] \scriptstyle\downarrow\\[-3pt] \text{\small $#2$}\end{matrix}\right)}}
\newcommand{\vto}[2]{{\begin{pmatrix}\text{\footnotesize $#1$}\\[-4.5pt] \scriptstyle\downarrow\\[-2.5pt] \text{\footnotesize $#2$}\end{pmatrix}}}

\tikzcdset{arrow style=tikz,diagrams={>=Classical TikZ Rightarrow}}

\renewcommand{\FunLSM}{\Fun^{\textup{lax-$\otimes$}}}

\begin{document}
{
\title{Homotopical commutative rings and bispans}
\date{\today}

\author{Bastiaan Cnossen,\textsuperscript{1} Rune Haugseng,\textsuperscript{2} Tobias Lenz,\textsuperscript{3} and Sil Linskens\textsuperscript{1}}

\maketitle

\footnotetext[1]{Mathematisches Institut, Universität Regensburg, Universitätsstraße 31, 93053 Regensburg, Germany}
\footnotetext[2]{Institutt for matematiske fag, Norges teknisk-naturvitenskapelige universitet, Alfred Getz’ vei 1, 7034 Trondheim, Norway}
\footnotetext[3]{Mathematical Institute, Universiteit Utrecht, Budapestlaan 6, 3584 CD Utrecht, The Netherlands \& Mathematisches Institut, Rheinische Friedrich-Wilhelms-Universität Bonn, Endenicher Allee 60, 53115 Bonn, Germany (\textit{current address})}

\begin{abstract}
  We prove that commutative semirings in a cartesian closed
  presentable $\infty$-category, as defined by Groth, Gepner, and
  Nikolaus, are equivalent to product-preserving functors from the
  $(2,1)$-category of bispans of finite sets. In other words, we
  identify the latter as the Lawvere theory for commutative semirings
  in the $\infty$-categorical context. This implies that connective
  commutative ring spectra can be described as grouplike
  product-preserving functors from bispans of finite sets to spaces. A
  key part of the proof is a localization result for
  $\infty$-categories of spans, and more generally for
  $\infty$-categories with factorization systems, that may be of
  independent interest.
\end{abstract}

\renewcommand{\thefootnote}{}
\footnotetext{2020 Mathematics Subject Classification: 18N70, 18N55}}

\begingroup\tosfstyle
\let\oldcontentsline=\contentsline
\def\contentsline#1#2#3#4{\oldcontentsline{#1}{#2}{\texttosf{#3}}{#4}}
\tableofcontents
\endgroup

\section{Introduction}

The goal of this paper is to give an explicit description of commutative semirings in the $\infty$-categorical setting in terms of a certain (2,1)-category $\Bispan(\xF)$ of \textit{bispans of finite sets}:

\begin{introthm}\label{mainthm}
	Let $\cC$ be a cartesian closed presentable \icat{}. Then there is a natural equivalence
	\[ \CRig(\cC) \simeq \Fun^{\times}(\Bispan(\xF), \cC)\]
	between the $\infty$-category of commutative semirings in $\cC$ and the $\infty$-category of product-preserving functors $\Bispan(\xF) \to \cC$.
\end{introthm}

Before properly introducing both sides of the equivalence, let us motivate this result by first discussing the analogous statement for \textit{commutative monoids}.

\subsubsection*{Commutative monoids and spans}
Given a category\footnote{In this paper we will call categories \emph{categories} and \icats{} \emph{\icats{}}. We hope this does not cause too much confusion for younger readers.} $\cC$ with finite products, there is a category $\CMon(\cC)$ of \textit{commutative monoids in $\cC$}, which consists of objects of $\cC$ equipped with a binary operation that is unital, associative and commutative. The theory of commutative monoids is algebraic in nature, in the sense that it can be described by a \textit{Lawvere theory}: the data of a commutative monoid in $\cC$ is the same as that of a product-preserving functor $L \to \Set$, where $L = \{\mathbb{N}^n \mid n \geq 0\}^{\mathrm{op}}$ is the opposite of the category of free commutative monoids on finitely many generators.

The Lawvere theory $L$ for commutative monoids admits an explicit description as the category $\hSpan(\xF)$ of spans of finite sets. Recall that the objects of this category are finite sets, while the set of morphisms in $\hSpan(\xF)$ from $S$ to $T$ is given as the set of isomorphism classes of \textit{spans} (or \textit{correspondences})
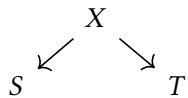
\[
\begin{tikzcd}[row sep=small, column sep=small]
	{} & X \ar[dl] \ar[dr] \\
	S & & T
\end{tikzcd}
\]
in finite sets. The identity maps are given by taking $S = X = T$, and composition is defined via pullback: given two composable spans $S \leftarrow X \to T$ and $T \leftarrow Y \to U$, their composite is given by the outer (dashed) span in the following pullback diagram:
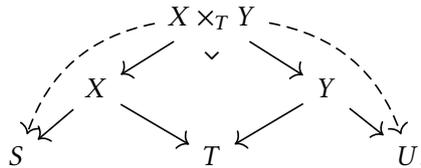
\[
\begin{tikzcd}[row sep=small, column sep=tiny]
	{} &[0.9em] {} & X \times_{T} Y \ar[dl] \ar[dr] \ar[ddll, bend right, dashed]
	\ar[ddrr, bend left, dashed] \arrow[phantom]{dd}[very near start]{\rotatebox{-45}{$\lrcorner$}} \\
	{} & X \ar[dl] \ar[dr] & & Y\ar[dl] \ar[dr] \\
	S & & T & &[0.9em] U.
\end{tikzcd}
\]
One can show that the category $\hSpan(\xF)$ admits finite products given by taking disjoint unions of sets.

The fact that $L$ is equivalent to $\hSpan(\xF)$ means that for every category $\cC$ with finite products there is a natural equivalence of categories
\[ \CMon(\cC) \simeq \Fun^{\times}(\hSpan(\xF), \cC),\]
where the right-hand side denotes the category of product-preserving functors from $\hSpan(\xF)$ to $\cC$. Explicitly, a commutative monoid $M$ in $\cC$ corresponds under this equivalence to the functor $M^{(-)}\colon \hSpan(\xF) \to \cC$ given as follows:
\begin{itemize}
	\item On objects, $M^{(-)}$ sends a finite set $S$ to the $S$-indexed product $M^{S} \cong M^{\times |S|}$;
	\item On morphisms, $M^{(-)}$ sends a span $S \xfrom{f} X \xto{g} T$ to the composite
	\[ M^{S} \xto{f^{*}} M^{X} \xto{g_{\oplus}} M^{T},\]
	where we define $f^*$ by $(f^*\psi)_x = \psi_{f(s)}$ for $\psi \in M^S$ and $x \in X$, and we define $g_{\oplus}$ by summing over the fibers of $g$: for $\phi \in M^X$ and $t \in T$ we set
	\[ (g_{\oplus}\phi)_t = \sum_{x \in g^{-1}(t)} \phi_x.\]
\end{itemize}

It turns out that the above description of commutative monoids carries over to the $\infty$-categorical setting, provided that we work with the $(2,1)$-category $\Span(\xF)$ in which we include isomorphisms of spans instead of taking
isomorphism classes. In other words, for any \icat{} $\cC$ with finite products there is a natural equivalence of $\infty$-categories
\[ \CMon(\cC) \simeq \Fun^{\times}(\Span(\xF), \cC).\]
This seems to have been first proved in the thesis of Cranch~\cite{CranchSpan}.

\subsubsection*{Commutative semirings and bispans}
In this paper we are interested in a similar description for the Lawvere theory for commutative \textit{semirings}. Recall that a commutative semiring in a category with finite products is an object $R$ equipped with two unital, associative and commutative operations, called \textit{addition} and \textit{multiplication}, satisfying the property that multiplication distributes over addition. The Lawvere theory of commutative semirings is given by the category $\hBispan(\xF)$ of \emph{bispans} of finite sets: the objects are still finite sets, but now the set of morphisms from $S$ to $T$ is given by the set of isomorphism classes of \emph{bispans} (or \emph{polynomial diagrams})
\[
\begin{tikzcd}
	{} & E \ar[r, "p"] \ar[dl, "s"{swap}] & B \ar[dr, "t"] \\
	S & & & T.
\end{tikzcd}
\]
The fact that there are now \textit{two} maps pointing to the right reflects the fact that we need to encode both the addition and the multiplication on $R$. Since these two operations interact via a distributivity relation, the composition rule in $\hBispan(\xF)$ is necessarily somewhat non-trivial to
describe. Given a bispan from $S$ to $T$ and bispan from $T$ to $U$, as displayed at the bottom of the following diagram, their composite is given by the outer (dashed) bispan in the diagram:
\begin{equation}
	\label{eq:bispancomp}
	\begin{tikzcd}[column sep=small]
		{} &   &   & G \arrow[dashed, bend left=30]{rrr} \arrow[dashed, bend
		right=30]{dddlll} \arrow{rr}[swap]{p''} \arrow{dl}{\epsilon'}
		\drpullback &
		& X \arrow[phantom]{ddr}[very near start,description]{\lrcorner}\arrow{r}[swap]{\tilde{q}} \arrow{dl}{\epsilon}
		\arrow{dd}{q^{*}q_{*}\pi}&[2.5em] D \arrow{dd}{q_{*}\pi} \arrow[dashed,bend left=15]{dddr} \\
		{} &   & Y \drpullback \arrow{rr}{p'} \arrow{dl}{u''} &   &  B \times_{J} F \arrow[phantom]{dd}[very near start]{\rotatebox{-45}{$\lrcorner$}} \arrow{dl}{u'} \arrow{dr}{\pi} \\
		{} & E\arrow[swap]{rr}{p} \arrow{dl}{s} &   & B \arrow[swap]{dr}{t} &     { }           &
		F \arrow{dl}{u}  \arrow[swap]{r}{q} & C \arrow[swap]{dr}{v} \\
		S  &   &   &   & T              &   &   & U.
	\end{tikzcd}
\end{equation}
Here four of the squares are pullbacks, as indicated, and $q_{*}$ is
the right adjoint to basechange along $q$, so that for
$f \colon K \to F$, the function $q_{*}f$ has fibres
$(q_{*}f)_{c} \cong \prod_{x \in q^{-1}(c)} K_{x}$; the map $\epsilon$
is the counit map $q^{*}q_{*} \to \id$ for the adjunction $q^{*} \dashv q_{*}$.

The statement that $\hBispan(\xF)$ is the Lawvere theory for
commutative semirings amounts to having, for any category $\cC$ with
finite products, a natural equivalence
\begin{equation}
	\label{eq:crigbispancat}
	\CRig(\cC) \simeq \Fun^{\times}(\hBispan(\xF), \cC)
\end{equation}
between the category of commutative semirings in $\cC$ and product-preserving functors from $\hBispan(\xF)$ to $\cC$.
We are unsure of where this statement was first proved, but it is discussed in Strickland's notes on Tambara functors \cite{StricklandTambara}*{\S 5}, and a proof appears in the work of Gambino and Kock on polynomial functors \cite{GambinoKock}.

Again, this equivalence can be easily described on objects; a commutative semiring $R$ in $\cC$ gets sent to the functor $R^{(-)}\colon\hBispan(\xF)\to\cC$ defined as follows:
\begin{itemize}
  \item On objects, $R^{(-)}$ sends a finite set $S$ to the $S$-indexed product $R^S$ as before.
  \item A morphism $S\xleftarrow{f} X\xrightarrow{g}Y\xrightarrow{h}T$ in $\hBispan(\xF)$ is sent to the composite
  \[
    R^S\xrightarrow{f^*} R^X\xrightarrow{g_\otimes}R^Y\xrightarrow{h_\oplus}R^T,
  \]
  where $f^*$ and $h_\oplus$ are given by restriction and fiberwise addition as before, while $g_\otimes$ is now given by fiberwise multiplication: if $\phi\in R^X$ is arbitrary, then
  \[
    g_\otimes(\phi)_y = \prod_{x\in g^{-1}(y)} \phi_x.
  \]
\end{itemize}

\subsubsection*{Commutative semirings in $\bm\infty$-categories}
The goal of this paper is to show that the description of the 1-categorical Lawvere theory of commutative semirings in terms of bispans generalizes to the $\infty$-categorical setting. For the definition of commutative semirings in this setting, we recall that Gepner, Groth, and
Nikolaus~\cite{GepnerGrothNikolaus} proved that if $\cC$ is a
presentable \icat{} equipped with a closed symmetric monoidal
structure, then the \icat{} $\CMon(\cC)$ of commutative monoids in
$\cC$ admits a unique closed symmetric monoidal structure such that the
free commutative monoid functor from $\cC$ to $\CMon(\cC)$ is symmetric
monoidal. We then define commutative semirings in $\cC$ as commutative
algebras with respect to this symmetric monoidal structure:
\[ \CRig(\cC):= \CAlg(\CMon(\cC), \otimes).\]

The Lawvere theory of commutative semirings will be a higher
category of bispans of finite sets, where we no longer take
isomorphism classes of bispans as above. To define this as an \icat{},
we make use of an observation due to Street~\cite{StreetPoly}: The
composition of \textit{bispans} can be described in terms of pullbacks in the category of \textit{spans}. This has been verified in the \icatl{} setting by
Elmanto and the second author \cite{bispans}, so that we can define
the $(2,1)$-category $\Bispan(\xF)$ as the $(2,1)$-category of spans in $\Span(\xF)$, where the
forward maps have no backward component.

Now that both sides have been defined, we recall the statement of our Theorem \ref{mainthm}: for every cartesian closed presentable $\infty$-category $\cC$, viewed as a symmetric monoidal category via the cartesian product, there is a natural equivalence
\[
\CRig(\cC) \simeq \Fun^{\times}(\Bispan(\xF), \cC).
\]
In other words we exhibit $\Bispan(\xF)$ as the Lawvere theory for commutative semirings in \icats{}. That this should be the case was suggested briefly at the end of \cite{BermanLawvere}; it was also proposed as a definition of commutative semirings in the thesis of Cranch~\cite{CranchThesis}, which contains the first (and rather different) construction of $\Bispan(\xF)$ as a
quasicategory.

Gepner--Groth--Nikolaus also show that the symmetric monoidal
structure on $\CMon(\cC)$ localizes to the full subcategory
$\CGrp(\cC) \subseteq \CMon(\cC)$ of \emph{grouplike} commutative
monoids, and that this is moreover compatible with the natural
symmetric monoidal structure on the stabilization of $\cC$. When
$\cC$ is the \icat{} $\Spc$ of spaces, this enhances the recognition
principle for infinite loop spaces of May~\cite{MayGeomIter} and
Boardman--Vogt~\cite{BoardmanVogt} to a symmetric monoidal equivalence
between $\CGrp(\Spc)$ and the \icat{} $\Sp^{\geq 0}$ of connective
spectra. Combining this with our result, we get a rather concrete
description of connective commutative ring spectra:
\begin{introcor}
	There is an equivalence
	\[ \CAlg(\Sp^{\geq 0}) \simeq \Fun^{\times}(\Bispan(\xF),
	\Spc)_{\mathrm{grp}} \] between connective commutative ring
	spectra and product-preserving functors $\Bispan(\xF) \to \Spc$
	whose underlying commutative monoid is grouplike.
\end{introcor}

\subsubsection*{Overview}
Let us outline our strategy for the proof of Theorem~\ref{mainthm}. First, we derive a more explicit description of $\CRig(\mathcal{C})$:
\begin{itemize}
	\item Recall that $\CMon(\cC)$ is equivalent to the full subcategory of
	$\Fun(\Span(\xF), \cC)$ spanned by the product-preserving functors. We identify the symmetric monoidal
	structure of \cite{GepnerGrothNikolaus} as a localization of the Day convolution structure arising from a tensor product on
	$\Span(\xF)$ given by the cartesian product of finite sets. This is essentially a result of Ben-Moshe and Schlank~\cite{BenMosheSchlank} (though they work in a slightly different setting).
	\item Using the universal property of Day convolution, this means that we can identify $\CRig(\cC)$ with a full subcategory of the
	\icat{} of lax symmetric monoidal functors
	$(\Span(\xF), \otimes) \to (\cC, \times)$.
	\item Using the universal property of the cartesian symmetric monoidal
	structure, we can identify this in turn as a full subcategory of the functors
	\[ \Span(\xF)^{\otimes} \to \cC, \] where $\Span(\xF)^{\otimes}$ is
	the total space of the cocartesian fibration that encodes the
	symmetric monoidal structure on $\Span(\xF)$.
\end{itemize}
We will prove this in \S\ref{sec:crig} (culminating with Corollary~\ref{cor:crigspanfot}), where we also give an explicit identification of the \icat{} $\Span(\xF)^{\otimes}$. Surprisingly, it ends up being another \icat{} of bispans, now in $\Ar(\xF)$ (Corollary~\ref{cor:spanFotimes}).

From this description of $\Span(\xF)^{\otimes}$ we obtain an evident
functor to $\Bispan(\xF)$. We are left with proving that this functor
is a localization, as well as showing that under this localization the
\icat{} of product-preserving functors from $\Bispan(\xF)$ corresponds
precisely to that of commutative semirings.  This is the content of
\S\ref{sec:comparison}. We deduce the first statement as a special case of a general recognition theorem for localizations which we believe may be of independent interest:

\begin{introthm}\label{introthm:sep-var}
	Let $f\colon\cC\to\cC'$ be a functor, and assume we have equipped $\cC$ and $\cC'$ with factorization systems $(E,M)$ and $(E',M')$, respectively, such that
	\begin{enumerate}
		\item $f$ restricts to a localization $E\to E'$ at some class $W\subset E$, and
		\item $f$ restricts to a right fibration $M\to M'$.
	\end{enumerate}
	Then $f\colon \cC\to \cC'$ is also a localization at $W$.
\end{introthm}

\subsubsection*{Generalizations}
In the follow-up work \cite{CHLL2}, we generalize the main theorem of this paper to the more general setting of \textit{parametrized higher algebra}. For simplicity, let us merely explain the statement in the $G$-equivariant setting for a finite group $G$. The generalization in this case involves replacing commutative monoids with \textit{$G$-normed monoids} (product-preserving functors $\Span(\xF_G) \to \cC$) and, correspondingly, symmetric monoidal $\infty$-categories with \textit{normed $G$-$\infty$-categories} ($G$-normed monoids in $\Cat_{\infty}$). Under the assumption that $\cC$ has colimits preserved by the cartesian product in both variables, the $\infty$-category of $G$-normed monoids in $\cC$ enhances to a normed $G$-$\infty$-category $\ul{\NMon}_G(\cC)$. This leads to the notion of \textit{$G$-normed semirings} in $\cC$,
\[
	\NRig_G(\cC) := \NAlg_G(\ul{\NMon}_G(\cC)),
\]
and the same strategy as Theorem~\ref{mainthm} shows an equivalence
\[
	\NRig_G(\cC) \simeq \Fun^\times(\Bispan(\xF_G),\cC),
\]
where $\Bispan(\xF_G)$ is the (2,1)-category of bispans of finite $G$-sets.

      Product-preserving functors from $\Bispan(\xF_G)$ to sets that
      are also grouplike in an appropriate sense are one definition
      of \emph{Tambara functors}, see e.g.\
      \cite{StricklandTambara}*{Definition 6.2}. Adding the same
      grouplike condition, we can therefore interpret this result as
      saying that the $\infty$-category of $G$-commutative rings in
      $\cC$ is equivalent to the $\infty$-category of
      $\cC$-valued $G$-Tambara functors. In particular, we
      obtain a description of connective $G$-commutative ring spectra
      as $G$-Tambara functors valued in spaces,
      \[ \NAlg_{G}(\smash{\ul{\Sp}}_G^{\geq 0}) \simeq \Fun^{\times}(\Bispan(\xF_{G}), \Spc)_{\mathrm{grp}}.\]

This can be seen as a multiplicative refinement of the comparison between (connective) genuine $G$-spectra and spectral $G$-Mackey functors (resp.~grouplike $G$-Mackey functors valued in spaces) from \cite{cmmn-Mackey}*{Theorem A.1}.

Although analogous to the main result of this paper, establishing this generalization requires substantial foundational work in parametrized higher category theory. We therefore defer its detailed treatment to the follow-up work \cite{CHLL2}.

\subsection*{Notation}
\begin{itemize}
	\item We write $\xF$ for the category of finite sets, and
	\[ \bfu{n} := \{1,\dots,n\}\]
	for a set with $n$ elements.
	\item We write $\CatI$ for the \icat{} of \icats{} and $\Spc$ for the
	\icat{} of spaces or \igpds{}.
	\item If $\cC$ is an \icat{}, then we usually denote its underlying \igpd{} by $\cC^{\simeq}$; in a few instances it will instead by notationally convenient to denote it by $\cC_{\eq}$, however.
	\item We denote generic \icats{} as $\cA, \cB, \cC,\dots$.
	\item We denote the arrow \icat{} of $\cC$ as $\Ar(\cC) := \Fun([1],
	\cC)$.
	\item A subcategory $\cC_{0}$ of an \icat{} $\cC$ is a functor
	$i \colon \cC_{0} \to \cC$ such that
	$\cC_{0}^{\simeq} \to \cC^{\simeq}$ and
	$\Map_{\cC_{0}}(x,y) \to \Map_{\cC}(i(x),i(y))$ for all $x,y$ are
	all monomorphisms of \igpds{}. In other words, subcategories are by
	definition always ``replete'', meaning that they must contain all
	equivalences among their objects. A subcategory is \emph{wide} if it
	contains all objects, or equivalently if
	$\cC_{0}^{\simeq} \to \cC^{\simeq}$ is an equivalence.
\end{itemize}

\subsection*{Acknowledgments}
A key idea for this paper arose in a discussion between Fabian
Hebestreit, Irakli Patchkoria and the second author during a visit to
the University of Aberdeen.

The authors would like to thank the anonymous referee for their detailed feedback.

The first author is an associate member of the SFB 1085 ‘Higher Invariants’ at the University of Regensburg, funded by the DFG. While the first version of this article was written, the fourth author was an associate member of the Hausdorff Center for Mathematics, supported by the DFG Schwer\-punktprogramm 1786 `Homotopy Theory and Algebraic Geometry' (project ID SCHW 860/1-1). On the other hand, during the time this article was revised, the third author was an associate member of the Hausdorff Center for Mathematics at the
University of Bonn (DFG GZ 2047/1, project ID 390685813).

\section{Background}
In this section we recall some background material: We review spans in
\S\ref{sec:spans} and bispans in \S\ref{sec:bispans}, and then look
briefly at symmetric monoidal \icats{} in general in
\S\ref{sec:symmon} and the special cases of (co)cartesian symmetric
monoidal structures in \S\ref{sec:cart}.

\subsection{Spans}\label{sec:spans}
In this subsection we briefly recall some definitions and results
relating to \icats{} of spans. These were originally constructed in
\cite{BarwickMackey} using quasicategories, though our primary
reference will be the model-independent reworking in \cite{HHLN2}.

\begin{defn}
  A \emph{span pair} $(\cC, \cC_{F})$
  consists of an \icat{} $\cC$ together with a wide
  subcategory $\cC_{F}$, whose
  morphisms we call the \emph{forward}
  morphisms, such that:
  \begin{enumerate}[(1)]
  \item for any forward morphism $f \colon x \to y$ and any
    morphism $g \colon z \to y$ there exists a pullback square
    \[
      \begin{tikzcd}
        w \ar[r, "f'"] \ar[d, "g'"'] & z \ar[d, "g"] \\
        x \ar[r, "f"'] & y,
      \end{tikzcd}
    \]
    in $\cC$,
  \item and in this pullback square the morphism $f'$ again lies in
    $\cC_{F}$.
  \end{enumerate}
  We write $\SPair$ for the \icat{} of span pairs; a morphism
  $(\cC, \cC_{F}) \to (\cD, \cD_{F})$ here is a
  functor $\cC \to \cD$ that preserves the forward maps
  as well as pullbacks along forward maps.

  In more detail, $\SPair$ is defined as the full subcategory of $\Fun([1],\CatI)$ spanned by the inclusions $\cC_F\hookrightarrow\cC$ for span pairs $(\cC,\cC_F)$. As each such functor is in particular a monomorphism, the evaluation map $\ev_{1}\colon\Fun([1],\CatI)\to\CatI$ is again a monomorphism, i.e.~a map between span pairs $(\cC, \cC_{F}) \to (\cD, \cD_{F})$ is indeed completely described by the functor $\cC\to\cD$ as claimed above, also cf.~\cite{HHLN2}*{2.1}.
\end{defn}

\begin{remark}
  The setup in \cite{BarwickMackey} and \cite{HHLN2} is a bit more
  general than this, and uses instead \emph{adequate
    triples} $(\cC, \cC_{F}, \cC_{B})$ where we specify subcategories
  of both forward and backward maps. Here we will only consider span
  pairs, as this simplifies the notation and encompasses all the
  examples we will encounter in this paper.
\end{remark}

\begin{observation}\label{obs:adtriplim}
  The \icat{} $\SPair$ has limits and filtered colimits, which are both computed in
  $\CatI$, by \cite{HHLN2}*{2.4}.
\end{observation}

\begin{ex}
  For any \icat{} $\cC$, we always have the span pair $(\cC,
  \cC^{\simeq})$.
\end{ex}

\begin{ex}
  If $\cC$ is an \icat{} with pullbacks, then $(\cC, \cC)$ is a
  span pair.
\end{ex}

\begin{notation}
  We write $\CatI^{\pb}$ for the subcategory of $\CatI$ consisting of
  \icats{} with pullbacks and functors that preserve these; the functor $\CatI^{\pb}\rightarrow \SPair, \cC\mapsto (\cC,\cC)$ identifies it as a
  full subcategory of $\SPair$ which is closed under limits and
  filtered colimits.
\end{notation}

Given a span pair $(\cC, \cC_{F})$, we can construct an \icat{}
\[\Span_{F}(\cC) = \Span(\cC, \cC_{F}),\] which informally has the same objects as $\cC$, with
a morphism from $x$ to $y$ given by a \emph{span} (or
\emph{correspondence})
\[
  \begin{tikzcd}
    {} & z \ar[dl,"b"{swap}] \ar[dr, "f"] \\
    x  & & y,
  \end{tikzcd}
\]
where $f$ is in $\cC_{F}$ and $b$ is arbitrary; composition is
given by taking pullbacks in $\mathcal{C}$. See \cite{HHLN2}*{2.12}
for a definition of $\Span_{F}(\cC)$ as a complete Segal space, which
gives a functor
\[ \Span \colon \SPair \to \CatI.\]
Given $\cC \in \CatIpb$, for the span pair $(\cC, \cC_{F})=(\cC, \cC)$ we will write
\[ \Span(\cC) := \Span_{F}(\cC).\]
We note some important properties of this functor:
\begin{itemize}
\item We have $\Span_{\eq}(\cC) = \Span(\cC, \cC^{\simeq}) \simeq \cC^{\op}$
  \cite{HHLN2}*{2.15}.
\item The functor $\Span$ preserves limits --- in fact, on the larger
  \icat{} of adequate triples it has a left
  adjoint, given by the twisted arrow \icat{} \cite{HHLN2}*{2.18}
\end{itemize}

\begin{observation}\label{obs:spaneq}
  A morphism in $\Span_{F}(\cC)$, that is a span
  \[
    \begin{tikzcd}[column sep=small, row sep=small]
      {} & Z \ar[dl, "f"'] \ar[dr, "g"] \\
      X & & Y
    \end{tikzcd}
  \]
  is invertible \IFF{} the maps $f$ and $g$ are invertible in $\cC$;
  see for instance \cite{spans}*{8.2} for a proof.
\end{observation}

Next, we recall two useful results relating spans and (co)cartesian fibrations. We begin with Barwick's \emph{unfurling construction}:

\begin{thm}[Barwick]\label{thm:unfurling}
  Suppose $\cC$ is an \icat{} with pullbacks and
  \[ \Phi \colon \cC^{\op} \to
  \CatI\] is a functor such that
  \begin{itemize}
  \item for every morphism $f \colon x \to y$ in $\cC$, the
    functor $f^{*}:= \Phi(f)$ has a left adjoint $f_{!}$,
  \item for every pullback square
        \[
      \begin{tikzcd}
        w \ar[r, "f'"] \ar[d, "g'"] & z \ar[d, "g"] \\
        x \ar[r, "f"] & y,
      \end{tikzcd}
    \]
    in $\cC$, the Beck--Chevalley transformation
    \[ f'_{!}g'^{*} \to g^{*}f_{!}\]
    is an equivalence.
  \end{itemize}
  Let $p \colon \cE \to \cC$ be the cartesian fibration for $\Phi$ and write $\cE_{\mathrm{cart}}$ for the subcategory of $\cE$ spanned by the $p$-cartesian edges. Then $(\cE,\cE_{\mathrm{cart}})$ is a span pair, and moreover the functor
 \[ \Span(p)^{\op} \colon \Span_{\mathrm{cart}}(\cE)^{\op} \to
    \Span(\cC)^{\op}\simeq \Span(\cC)\] is a cocartesian fibration for a functor
  $\Span(\cC) \to \CatI$ that restricts to $\Phi$ on $\cC^{\op}$ and
  is given on forward maps by taking left adjoints.
\end{thm}

\begin{proof}
  This is a special case of \cite{BarwickMackey}*{11.6}; see also
  \cite{HHLN2}*{3.2 and 3.4} for further discussion.
\end{proof}

\begin{remark}
  In the situation of Theorem~\ref{thm:unfurling}, the universal property of
  the \itcat{} $\SPAN(\cC)$ of spans in $\cC$ \cite{MacphersonCorr, Stefanich} says that there is a
  unique functor of \itcats{} $\SPAN(\cC) \to \CATI$ that extends
  $\Phi$. By the main result of \cite{CLR} its underlying functor of \icats{}
  corresponds to the cocartesian fibration
  $\Span(p)^{\op}$.
\end{remark}

\begin{thm}\label{thm:fibforspan}
  Suppose $p \colon \cE \to \cC$ is the cartesian fibration for a
  functor $F \colon \cC^{\op} \to \CatIpb$. Then the cocartesian
  fibration for the composite functor $\Span \circ F \colon \cC^{\op}
  \to \Cat$ is
  \[ \Span(p) \colon \Span_{\fw}(\cE) \to
    \Span_{\eq}(\cC) \simeq \cC^{\op},\] where
  $\cE_{\fw}$ contains the (``fibrewise'') maps in $\cE$ that map to
  equivalences in $\cC$ under $p$.
\end{thm}
\begin{proof}
  This is a special case of \cite{HHLN2}*{3.9}.
\end{proof}

\subsection{Bispans}\label{sec:bispans}
We now recall the definition of \icats{} of \emph{bispans}, following
\cite{bispans}.

\begin{defn}[\cite{bispans}*{2.4.3 and 2.4.6}]\label{def:bispantrip}
  A \emph{bispan triple} $(\cC, \cC_{F}, \cC_{L})$ consists of an
  \icat{} $\cC$ together with two wide subcategories $\cC_{F}$ and
  $\cC_{L}$ such that the following conditions hold:
  \begin{enumerate}[(1)]
  \item Both $(\cC, \cC_{F})$ and $(\cC, \cC_{L})$ are span pairs.
  \item\label{it:bispanradj} Let $\cC_{/x}^{L} \subseteq \cC_{/x}$ be the full subcategory
    spanned by the maps to $x$ that lie in $\cC_{L}$; for $f \colon x
    \to y$ in $\cC_{F}$, the functor
    $f^{*} \colon \cC_{/y}^{L} \to \cC_{/x}^{L}$ given by basechange along $f$ has a right adjoint
    $f_{*}$.
  \item\label{it:bispanmate} For every pullback square
        \[
      \begin{tikzcd}
        x' \ar[r, "f'"] \ar[d, "\xi"{swap}] & y' \ar[d, "\eta"] \\
        x \ar[r, "f"'] & y
      \end{tikzcd}
    \]
    with $f$ in $\cC_{F}$, the commutative square
    \[
      \begin{tikzcd}
        \cC^{L}_{/y} \ar[r, "f^{*}"] \ar[d, "\eta^{*}"'] & \cC^{L}_{/x}
        \ar[d, "\xi^{*}"] \\
        \cC^{L}_{/y'} \ar[r, "f'^{*}"'] & \cC^{L}_{/x'}
      \end{tikzcd}
    \]
    is right adjointable, \ie{} the mate transformation $\eta^{*}f_{*} \to
    f'_{*}\xi^{*}$ is an equivalence.
  \end{enumerate}
\end{defn}

\begin{remark}
  If $\cC_{L} = \cC$, then condition \ref{it:bispanradj} says
  precisely that $\cC$ is locally cartesian closed; in this case, condition
  \ref{it:bispanmate} is automatic.
\end{remark}

\begin{thm}[\cite{bispans}*{2.5.2(1)}]
  Suppose $(\cC, \cC_{F})$ is a span pair and suppose $\cC_{L}$ is a
  wide subcategory of $\cC$. Then $(\cC, \cC_{F}, \cC_{L})$ is a
  bispan triple \IFF{} $(\Span_{F}(\cC)^{\op}, \cC_{L})$ is a span
  pair, where we regard $\cC_{L}$ as contained in the subcategory
  $\cC \simeq \Span_{\mathrm{eq}}(\cC)^{\op}$ inside
  $\Span_{F}(\cC)^{\op}$.\qed
\end{thm}

\begin{defn}
  Suppose $(\cC, \cC_{F}, \cC_{L})$ is a bispan triple. Then we define
  \[ \Bispan_{F,L}(\cC) := \Span_{L}(\Span_{F}(\cC)^{\op}).\] If
  $\cC_{L} = \cC$ we abbreviate this to $\Bispan_{F}(\cC)$, and if
  also $\cC_{F} = \cC$ we just write $\Bispan(\cC)$.
\end{defn}

Explicitly, this means that a generic morphism in $\Bispan_{F,L}(\cC)$ looks as follows:
\[
  \begin{tikzcd}
    & \arrow[dl,"\in\cC"'] B\arrow[dr,"\in\cC_F"] && C\arrow[dr, "\in\cC_L"]\\
    A && C\arrow[ur,equals] && D\rlap;
  \end{tikzcd}
\]
here the span $A\gets B\to C$ defines a map from $A$ to $C$ in $\Span_F(\cC)$, and hence a map from $C$ to $A$ in $\Span_F(\cC)^\op$. We will typically denote such a generic morphism by
\[
  \begin{tikzcd}
    A & \arrow[l, "\in\cC"'] B \arrow[r,"\in\cC_F"] & C\arrow[r,"\in\cC_L"] & D\rlap.
  \end{tikzcd}
\]

\begin{remark}
  The fact that we get the correct composition of bispans by viewing
  them as ``spans in spans'' was first observed by
  Street~\cite{StreetPoly}. The first (and rather different)
  construction of $\Bispan(\xF)$ as an \icat{} is in the thesis of
  Cranch~\cite{CranchThesis}.
\end{remark}

\begin{defn}\label{def:bispanmor}

  If $(\cC, \cC_{F}, \cC_{L})$ and $(\cD, \cD_{F}, \cD_{L})$ are bispan triples, then a
  \emph{morphism of bispan triples} is a functor $\Phi \colon \cC \to \cD$ that induces morphisms of span pairs
  $(\cC, \cC_{F}) \to (\cD, \cD_{F})$ and
  $(\Span_{F}(\cC)^{\op}, \cC_{L}) \to (\Span_{F}(\cD)^{\op},
    \cD_{L})$.

  Assuming the first condition, the second one is equivalent to $(\cC, \cC_{L}) \to (\cD, \cD_{L})$ being a morphism of span pairs and the square
  \[
    \begin{tikzcd}
      \cC^{L}_{/y} \ar[r, "f^{*}"] \ar[d, "\Phi"] & \cC^{L}_{/x} \ar[d, "\Phi"'] \\
      \cD^{L}_{/\Phi(y)} \ar[r, "\Phi(f)^{*}"'] & \cD^{L}_{/\Phi(x)}
    \end{tikzcd}
  \]
  being right adjointable for every morphism $f \colon x \to y$ in $\cC_{F}$, \ie{} the Beck--Chevalley transformation
  \[ \Phi \circ f_{*} \to \Phi(f)_{*} \circ \Phi\] is an equivalence; this follows from the identification of pullbacks in $\Span_F(\cC)^{\op}$ of maps $\cC \subset \Span_F(\cC)^{\op}$ along forwards and backwards morphisms in $\Span_F(\cC)^{\op}$ with pullbacks in $\cC$ and \emph{distributivity
    diagrams} in $\cC$ respectively, see~\cite{bispans}*{2.5.10, 2.5.12}.
\end{defn}

\subsection[Symmetric monoidal \icats{}]{Symmetric monoidal $\bm\infty$-categories}\label{sec:symmon}
In this subsection we recall some definitions related to commutative
monoids and symmetric monoidal structures on \icats{}.

\begin{defn}\label{def:cmon}
  Let $\cC$ be an \icat{} with finite
  products. A \emph{commutative monoid} in $\cC$ is a functor
  \[ M \colon \Span(\xF) \to \cC\]
  that preserves finite products. We write $\CMon(\cC)$ for the
  full subcategory of $\Fun(\Span(\xF), \cC)$ spanned by the
  commutative monoids.
\end{defn}

\begin{remark}
  This definition of commutative monoids fits into the framework for
  algebraic structures defined by Segal conditions from
  \cite{patterns1}: We can endow $\Span(\xF)$ with the structure of an
  \emph{algebraic pattern} where the inert--active factorization
  system is that given by the backwards and forwards maps, and the
  point is the only
  elementary object. Then a Segal
  $\Span(\xF)$-object in $\cC$ is a functor $M$ such that
  \[ M(\bfu{n}) \isoto \lim_{(\{\bfone\}_{/\bfu{n}})^{\op}} M(\bfone) \simeq
    \prod_{i=1}^{n} M(\bfone)\]
  for every $n$.
\end{remark}

\begin{remark}
  This definition of commutative monoids is equivalent to that used in
  \cite{HA} in terms of finite pointed sets; see for instance
  \cite{harpaz}*{5.14} or \cite{norms}*{C.1}.
\end{remark}

\begin{defn}
  A symmetric monoidal \icat{} is a commutative monoid in $\CatI$; its
  \emph{underlying \icat{}} is the value at $\bfone$. Given a
  symmetric monoidal structure on an \icat{} $\cC$, we
  will denote the corresponding cocartesian and cartesian fibrations
  by\footnote{Here $\Span(\xF)^{\op} \simeq \Span(\xF)$, but we write
    $\op$ as a reminder that this is a cartesian fibration.}
  \[ \cC^{\otimes} \to \Span(\xF), \quad \cC_{\otimes} \to
    \Span(\xF)^{\op}.\]
  We say that a morphism in $\cC^{\otimes}$ is \emph{inert} if
  it is cocartesian over a backwards morphism in $\Span(\xF)$;
  similarly, a morphism in $\cC_{\otimes}$ is \emph{inert} if it is
  cartesian over a (reversed) backwards morphism in
  $\Span(\xF)^{\op}$.
\end{defn}

\begin{defn}
  Suppose $\cC^{\otimes}, \cD^{\otimes} \to
  \Span(\xF)$ are symmetric monoidal \icats{}. A \emph{symmetric
    monoidal functor} from $\cC$ to $\cD$ is a
  commutative triangle
  \[
    \begin{tikzcd}
      \cC^{\otimes} \ar[rr, "F"] \ar[dr] & &
      \cD^{\otimes} \ar[dl] \\
       & \Span(\xF)
    \end{tikzcd}
  \]
  where $F$ preserves cocartesian morphisms. We say that $F$ is
  \emph{lax symmetric monoidal} if it instead only preserves inert
  morphisms. We write
  \[\FunLSM(\cC^\otimes, \cD^\otimes)\coloneqq\Fun_{/\Span(\xF)}^\textup{inert}(\cC^{\otimes}, \cD^{\otimes})\] for
  the full subcategory of
  $\Fun_{/\Span(\xF)}(\cC^{\otimes}, \cD^{\otimes})$ spanned by the
  lax symmetric monoidal functors.
\end{defn}

\begin{remark}
  It follows from \cite{envelopes}*{5.1.15} that this
  definition of lax symmetric monoidal functors agrees with the more
  standard one, with $\xF_{*}$ in place of $\Span(\xF)$. By \cite{anmnd2}*{2.2.7 and 2.2.10}, we can also
  equivalently define a lax symmetric monoidal functor to be a
  commutative triangle as above where $F$ preserves finite products.
\end{remark}

\begin{defn}
  A \emph{commutative algebra} in a symmetric monoidal \icat{}
  $\cC^{\otimes}$ is a lax symmetric monoidal functor from
  $*^\otimes=\Span(\xF)$; we write
  \[ \CAlg(\cC) \coloneqq \FunLSM(*^\otimes,
    \cC^\otimes) =\Fun_{/\Span(\xF)}^\txt{inert}(\Span(\xF),\cC^\otimes) \]
  for the \icat{} of commutative algebras in $\cC$.
\end{defn}

\subsection{Cartesian and cocartesian symmetric monoidal structures}\label{sec:cart}
In this subsection we review some results from \cite{HA} on cartesian and
cocartesian symmetric monoidal structures, meaning those arising from
products and coproducts in an \icat{}. Let us first recall a precise definition of such structures:
\begin{defn}[\cite{HA}*{2.4.0.1}]\label{defn:cartsymmon}
  Let $\cC^{\otimes}$ be a symmetric monoidal \icat{}. We say that it is
  \emph{cocartesian}
  if
  \begin{itemize}
  \item the unit $\bbone$ in $\cC$ is initial,
  \item for all objects $X,Y \in \cC$, the maps
    \[ X \simeq X \otimes \bbone \to X \otimes Y \from \bbone \otimes
      Y \simeq Y\]
    exhibit $X \otimes Y$ as the coproduct of $X$ and $Y$.
  \end{itemize}
  Dually, $\cC^{\otimes}$ is \emph{cartesian} if
  \begin{itemize}
  \item the unit $\bbone$ in $\cC$ is terminal,
  \item for all objects $X,Y \in \cC$, the maps
    \[ X \simeq X \otimes \bbone \from X \otimes Y \to \bbone \otimes
      Y \simeq Y\]
    exhibit $X \otimes Y$ as the product of $X$ and $Y$.
  \end{itemize}
\end{defn}

\begin{remark}\label{rk:cocart-adj}
  As the unit $\bbone_{\cC}$ is simply the image of the essentially
  unique object of $\cC(\bfn0)\simeq*$ under the functor
  $\cC(\bfn0=\bfn0\to\bfn1)$, in the cocartesian case the first condition is equivalent to
  demanding that $\cC(\bfn0=\bfn0\to\bfn1)$ is a left
  adjoint. If in this case also $\cC(\bfn2=\bfn2\to\bfn1)$ is a left adjoint, then
  the second condition follows, because
  $(X,\bbone)\to (X,Y)\gets (\bbone,Y)$ is a coproduct diagram in
  $\cC(\bfn2)$. Conversely, it will follow from Proposition~\ref{propn:spancocart} that all forward maps in a cocartesian symmetric monoidal structure are left adjoint to the corresponding backwards map.
\end{remark}

\begin{thm}[Lurie]\label{thm:cartuniv}
  Suppose $\cC$ is an \icat{} with finite coproducts. Then there
  is a \emph{unique} cocartesian symmetric monoidal \icat{}
  $\cC^{\amalg}$ such that $\cC^{\amalg}_{\bfone}
  \simeq \cC$. Dually, if $\cC$ is an \icat{} with finite products, then there is a \emph{unique} cartesian symmetric monoidal \icat{} $\cC^{\times}$ such that $\cC^{\times}_{\bfone} \simeq \cC$.
\end{thm}
\begin{proof}
  This is part of \cite{HA}*{2.4.1.8 and 2.4.3.12}.
\end{proof}

In the cartesian case, we can further describe lax symmetric monoidal functors to $\cC^{\times}$ in terms of certain functors to $\cC$:
\begin{defn}\label{defn:monoid}
  Suppose $\cD^{\otimes} \to \Span(\xF)$ is a symmetric
  monoidal \icat{} and $\cC$ is an \icat{} with finite
  products. A \emph{$\cD^{\otimes}$-monoid} in $\cC$
  is a functor
  \[ M \colon \cD^{\otimes} \to \cC\]
  such that for every object $X \in \cD^{\otimes}$ over $\bfn{n} \in
  \Span(\xF)$, the map
  \[ M(X) \isoto \prod_{i=1}^{n} M(X_i)\]
  is an equivalence, where $X\to X_i$ for $i=1,\dots,n$ denotes a cocartesian lift of the span
  \[
    \textbf{n}\xleftarrow{\;i\;}\bfone\xrightarrow{\;=\;}\bfone.
  \]
  We write
  $\Mon_{\cD^{\otimes}}(\cC)$ for the full subcategory
  of $\Fun(\cD^{\otimes}, \cC)$ spanned by the $\cD^{\otimes}$-monoids.
\end{defn}

\begin{remark}
    By \cite{anmnd2}*{2.3.3}, we can equivalently
  characterize $\cD^{\otimes}$-monoids in $\cC$ as functors
  $\cD^{\otimes} \to \cC$ that preserve finite products.
\end{remark}

\begin{thm}[Lurie]\label{thm:lsm=monoid}
  Suppose $\cC$ is an \icat{} with finite products. Then there is a
  functor $\cC^{\times} \to \cC$ that induces an equivalence
  \[ \FunLSM(\mathcal{D}^{\otimes},
    \cC^{\times}) \isoto
    \Mon_{\mathcal{D}^{\otimes}}(\cC)\]
  between lax symmetric monoidal functors and monoids.
\end{thm}

While we will never need to know the functor $\cC^\times\to\cC$ explicitly, let us record its effect on objects for motivational purposes: it is given by sending a tuple $(X_1,\dots,X_n)\in\cC^{\times n}\simeq(\cC^\times)_\textbf{n}$ to the product $X_1\times\cdots\times X_n$.

\begin{proof}
  This is the content of \cite{HA}*{\S 2.4.1}, translated through the
  equivalence between \iopds{} over $\xF_{*}$ and $\Span(\xF)$ from
  \cite{envelopes}.
\end{proof}

Applied to $\mathcal{D}^\otimes = \ast^\otimes$, the previous theorem gives the following corollary:

\begin{cor}
Let $\cC$ be an \icat{} with finite products. Then the functor $\cC^{\times} \to \cC$ from Theorem~\ref{thm:lsm=monoid} induces an equivalence $\CAlg(\cC^\times)\simeq \CMon(\cC)$.
\end{cor}

\section{Commutative semirings and spans}\label{sec:crig}
Our goal in this section is to give an explicit description of
commutative semirings, defined as commutative algebras in the
symmetric monoidal structure on commutative monoids constructed by
Gepner--Groth--Nikolaus. We first obtain a concrete construction of
the fibration $\Span(\xF)^{\otimes} \to \Span(\xF)$ for the symmetric
monoidal structure on spans induced by the cartesian product of finite
sets, by first describing (co)cartesian symmetric monoidal structures
in terms of spans in \S\ref{sec:cartspan} and then describing
symmetric monoidal structures on \icats{} of spans in
\S\ref{sec:spansymmon}. In \S\ref{sec:dayconv} we then prove, following
Ben-Moshe and Schlank, that the symmetric monoidal structure on
commutative monoids is a localization of the Day convolution for this
symmetric monoidal structure on $\Span(\xF)$; from this we then get
the desired description of commutative semirings in $\cC$ as
lax symmetric monoidal functors from $\Span(\xF)$.

\subsection{(Co)cartesian symmetric monoidal structures via spans}\label{sec:cartspan}
In this subsection we will see that the cocartesian fibration for a
cocartesian symmetric monoidal structure can be described in terms of
spans. Moreover, so can the cartesian fibration for a cartesian
symmetric monoidal structure, in a way that is \emph{not} simply dual
to the first description.\footnote{Such a dual description does also
  exist, but this is in terms of \emph{co}spans rather than spans.}
\begin{propn}\label{propn:spancocart}
  Suppose $\cC$ is an \icat{} with finite coproducts, and let
  \[p \colon \cC_{\xF} \to \xF\] denote the cartesian fibration for the
  functor
  \[
    \cC^{(-)}\colon \xF^{\op}\rightarrow \Cat_\infty, \quad S \mapsto \Fun(S, \cC).
  \]
  Then $\Span_{\cart}(\cC_{\xF})^{\op} \to \Span(\xF)^{\op}\simeq \Span(\xF)$ is
  the cocartesian fibration for the cocartesian symmetric monoidal
  structure on $\cC$.
\end{propn}
\begin{proof}
  Because $\cC$ admits coproducts, the functor
  $\cC^{(\blank)}\colon \xF^{\op}\rightarrow \Cat$ satisfies the assumptions of
  Theorem~\ref{thm:unfurling} (see e.g.~\cite{HopkinsLurie}*{4.3.3} for the Beck--Chevalley condition), showing that $(\cC_{\xF},\cC_{\xF,\cart})$ is a span
  pair and that the functor
  \[ \Span_{\cart}(\cC_{\xF})^{\op} \to \Span(\xF)\] is a cocartesian fibration.

  Furthermore, as explained there, the restriction of the
  corresponding functor $\Span(\xF)\rightarrow \Cat$ to
  $\xF^{\op}\subset \Span(\xF)$ agrees with the original functor
  $\cC^{(-)}$ (and so $\Span_{\cart}(\cC_{\xF})^{\op}\to\Span(\xF)$ defines a symmetric
  monoidal structure on $\cC$), while the restriction to
  $\xF\subset \Span(\xF)$ agrees with the functor obtained from
  $\cC^{(-)}$ by passing to left adjoints. In particular, it is
  cocartesian monoidal by Remark~\ref{rk:cocart-adj}.
\end{proof}

We have a similarly explicit description of the
cartesian fibration for a cartesian symmetric monoidal structure:
\begin{propn}\label{lem:Span(CF)cart}
  Suppose $\cC$ is an \icat{} with finite products. Then
  \[\Span(p)^\op\colon \Span_{\cart}(\cC_{\xF})^{\op} \to \Span(\xF)^{\op}\] is a
  cartesian fibration, which classifies the cartesian symmetric
  monoidal structure on $\cC$. That is,
  $\Span_{\cart}(\cC_{\xF})^{\op} \simeq \cC_{\times}$.
\end{propn}

Note that this in particular means that if $\cC$ has both coproducts and products, then the \emph{cartesian} fibration for the cartesian symmetric monoidal structure and the \emph{cocartesian} fibration for the cocartesian structure differ only by postcomposition with the canonical equivalence $\Span(\xF)^\op\simeq\Span(\xF)$. This observation is in fact crucial for our proof:

\begin{proof}
  Assume first that $\cC$ in addition has finite coproducts, so we can apply
  Proposition~\ref{propn:spancocart} to describe the cocartesian unstraightening
  of the \emph{co}cartesian symmetric monoidal structure on $\cC$ as the composite
  \[
    q\colon \Span_{\cart}(\cC_{\xF})^\op\xrightarrow{\Span_F(p)^\op}\Span(\xF)^\op\isoto\Span(\xF).
  \]
    The corresponding functor $\Span(\xF)\to\CatI$ then sends a span
  \begin{equation}\label{eq:generic-span}
    S\xleftarrow{g} X\xrightarrow{f}T
  \end{equation}
  to the composite
  \[
    \Fun(S,\cC)\xrightarrow{g^*} \Fun(X,\cC)\xrightarrow{f_{\kern.5pt!}} \Fun(T,\cC),
  \]
  where $f_{\kern.5pt!}$ is left adjoint to $f^*$. Because $\cC$ also admits finite products, each of these functors is a left adjoint,
  so $q\colon\Span_{\cart}(\cC_{\xF})^{\op} \to\Span(\xF)$ is also a cartesian
  fibration. We may now straighten it \emph{as a cartesian fibration}, yielding the functor $\Span(\xF)^{\op}\rightarrow \CatI$ obtained by passing to
  right adjoints, i.e.~sending the generic span (\ref{eq:generic-span}) to $g_*f^*\colon\Fun(T,\cC)\to\Fun(S,\cC)$, where $g_*$ is right adjoint to $g^*$. It follows that also the composite $\Span_{\cart}(\cC_{\xF})^\op \xrightarrow{\smash{q}} \Span(\xF)\simeq \Span(\xF)^{\op}$ (which is just $\Span(p)^{\op}$) is a cartesian fibration, whose cartesian straightening $F\colon\Span(\xF)\to\CatI$ then sends a span (\ref{eq:generic-span}) to $f_*g^*$ (the effect of the reversed span). The restriction of $F$ to $\xF^\op$ is the original functor $\Fun(-,\cC)\colon\xF^\op\to\CatI$ as we have taken left adjoints and then right adjoints. Thus, $F$ also preserves products, as this condition only
  depends on the backwards maps, and so
  $\Span_{\cart}(\cC_{\xF}) \to\Span(\xF)$ is the cartesian
  unstraightening of a symmetric monoidal $\infty$-category. By the
  dual of Remark~\ref{rk:cocart-adj} it is moreover cartesian. This
  completes the proof under the additional assumption that $\cC$ also
  has finite coproducts.

  For general $\cC$, let $\mathcal D$ denote the finite-coproduct-completion
  of $\cC$, \ie{} the full subcategory of the \icat{} of presheaves on $\cC$ spanned
  by finite coproducts of representables. As $\cC$ has finite products, and since coproducts
  and products distribute in spaces, $\mathcal D$ also has finite products, and the Yoneda
  embedding $\cC\hookrightarrow\mathcal D$ preserves finite products.

  By the above special case, $\Span_{\cart}(\mathcal D_{\xF})^\op\to\Span(\xF)$
  is then the cartesian fibration for the cartesian symmetric monoidal on on $\mathcal D$.
  As the inclusion $\cC\hookrightarrow\mathcal D$ preserves products, it extends to
  a fully faithful functor $\cC_\times\hookrightarrow\mathcal D_\times=\Span_{\cart}(\mathcal D_{\xF})^{\op}$ with essential image those
  tuples $(X_1,\dots,X_n)$ such that all $X_i$ are contained in $\cC$. As the
  natural inclusion $\Span_{\cart}(\cC_{\xF})^\op\to\Span_{\cart}(\mathcal D_{\xF})^\op$ is also
  fully faithful with the same essential image, the claim follows.
\end{proof}

We will now specialize these results to the cocartesian and cartesian symmetric monoidal structures on $\xF$:

\begin{cor}\label{cor:Fcoprod}
  Let $\Ar(\xF)_{\pb}$ denote the subcategory of $\Ar(\xF)$ whose
  morphisms are the commutative squares that are pullbacks.  Then
  \[\ev_{1} \colon \Span_{\pb}(\Ar(\xF))^{\op} \to \Span(\xF)\]
  is the cocartesian fibration for the cocartesian symmetric monoidal
  structure on $\xF$.
\end{cor}
\begin{proof}
  By Proposition~\ref{propn:spancocart} it suffices to show that $\ev_{1} \colon
  \Ar(\xF) \to \xF$ is the cartesian fibration for the functor
  $\xF^{(\blank)} \colon \xF^{\op} \to \CatI$. But $\ev_{1}$ is the cartesian fibration for $S \mapsto \xF_{/S}$, with
  functoriality given by basechange, so this follows from the natural
  straightening equivalence $\xF_{/S} \simeq \Fun(S, \xF)$, under
  which pullbacks correspond to compositions.
\end{proof}

In the same way we get:

\begin{cor}\label{cor:Ftimesspan}
  There is an equivalence of cartesian fibrations
  \[ \Span_{\pb}(\Ar(\xF))^{\op} \simeq \xF_{\times}\] over
  $\Span(\xF)^{\op}$ (using evaluation at $1$ in the arrow category).
\end{cor}

\subsection{Symmetric monoidal structures on spans}\label{sec:spansymmon}
In this subsection we discuss symmetric monoidal structures on span
\icats{}, and in particular on $\Span(\xF)$. These symmetric monoidal
structures were first constructed in \cite{BarwickMackey2}.

\begin{construction}\label{const:sym_mon_spans}
  Since the functor $\Span \colon \SPair \to \CatI$ preserves limits,
  and in particular finite products, it also preserves commutative
  monoids. Therefore, $\Span$ applied to any commutative monoid in
  $\SPair$ (or more generally in adequate triples) is a symmetric
  monoidal \icat{}. Here we will only consider this construction for
  objects of $\CatI^{\pb} \subseteq \SPair$. A commutative monoid in
  $\CatI^{\pb}$ is a symmetric monoidal \icat{}
  $M\colon \Span(\xF)\to \CatI$ such that
  \begin{itemize}
  \item the underlying \icat{} $\cC = M(\bfone)$ has pullbacks,
  \item the tensor product functor
    $\blank\otimes \blank \colon \cC\times \cC\rightarrow \cC$
    preserves pullbacks.\footnote{Here we mean pullbacks in
      $\cC\times \cC$, so that the tensor product of a pair of pullback squares in $\cC$ is again a pullback. This condition implies that $\otimes$ preserves pullbacks in each variable (since the pullback shape is weakly contractible), but is strictly stronger than this.}
  \end{itemize}
 In this case,
  $\Span(M) \colon \Span(\xF) \to \CatI$ is a commutative monoid,
  which endows $\Span(\cC)$ with a symmetric monoidal structure whose
  tensor product is inherited from $\cC$.
\end{construction}

\begin{ex}
  If $M\colon \Span(\xF)\to \Cat_\infty$ is the cartesian symmetric
  monoidal structure on an \icat{} $\cC$ with finite limits, then
  $ \blank \times \blank \colon \cC\times \cC\to \cC$ clearly
  preserves pullbacks, and so induces a symmetric monoidal structure
  on $\Span(\cC)$.\footnote{Note that this is \emph{not} the cartesian
    monoidal structure on $\Span(\cC)$!}  For example, we conclude
  that $\Span(\xF)$ is endowed with a symmetric monoidal structure
  given by the cartesian product in $\xF$.
\end{ex}

We can describe the cocartesian fibrations for these symmetric monoidal structures rather explicitly:
\begin{propn}\label{propn:spansmon}
  Suppose $p \colon \cC_{\otimes} \to \Span(\xF)^{\op}$ is the
  cartesian fibration for a commutative monoid in $\CatIpb$. Then the
  cocartesian fibration $\Span(\cC)^{\otimes} \to \Span(\xF)$ for the
  induced symmetric monoidal structure on spans from
  Construction~\ref{const:sym_mon_spans} is given by
  \[ \Span(\cC)^{\otimes} \simeq
    \Span_{\fw}(\cC_{\otimes}),\]
  where $(\cC_{\otimes})_{\fw}$ denotes the subcategory
  of maps that go to equivalences under $p$.
\end{propn}

\begin{proof}
  This follows immediately from Theorem~\ref{thm:fibforspan}.
\end{proof}

\begin{remark}
  When the symmetric monoidal structure on $\cC$ is cartesian, the
  description of $\Span(\cC)^\otimes$ above shows that it agrees with the symmetric monoidal structure on $\Span(\cC)$ constructed as a special case of \cite{BarwickMackey2}*{2.14}.
\end{remark}

Combining Proposition~\ref{propn:spansmon} with the description of $\xF_{\times}$
from Corollary~\ref{cor:Ftimesspan}, we get the
following special case:
\begin{cor}\label{cor:spanFotimes}
  The composite
  \begin{multline*}
    \Bispan_{\pb,\teq}(\Ar(\xF)) =
    \Span_{\teq}(\Span_{\pb}(\Ar(\xF))^{\op})\\
    {}\xrightarrow{\ev_1}\Span_{\eq}(\Span(\xF)^\op)\simeq\Span(\xF)
  \end{multline*}
  is a cocartesian fibration, where $\Ar(\xF)_{\pb}$ denotes the subcategory of $\Ar(\xF)$ where
  the morphisms are pullback squares, and
  $\Span(\Ar(\xF))_{\teq}$ denotes the subcategory of
  morphisms whose image under $\ev_{1}$ is a span of equivalences in
  $\xF$. Moreover, this is the cocartesian unstraightening $\Span(\xF)^\otimes\to\Span(\xF)$ of the symmetric monoidal structure on $\Span(\xF)$ induced by the cartesian structure on $\cC$.
  \qed
\end{cor}

\begin{remark}
  A morphism in $\Span(\xF)^{\otimes}$ is then a diagram
  \[
    \begin{tikzcd}
      X \arrow{d} & Y \arrow{l} \arrow{r} \arrow{d} \drpullback & X' \arrow{d} & Y' \arrow{l}{\sim} \arrow{r} \arrow{d} \dlpullback & X'' \arrow{d} \\
      S & T \arrow{l}{f} \arrow{r}[swap]{g} & S' & T' \arrow{l}{\sim} \arrow{r}[swap]{\sim} & S''
    \end{tikzcd}
  \]
  which simplifies to
  \[
    \begin{tikzcd}
      X \arrow{d} & Y \arrow{l} \arrow{r} \arrow{d} \drpullback & X' \arrow{d} \arrow{r} & X'' \arrow{dl} \\
      S & T \arrow{l}{f} \arrow{r}[swap]{g} & S'.
    \end{tikzcd}
  \]
  This amounts to specifying a family of spans
  \[\prod_{t \in g^{-1}(s')} X_{f(t)} \from X'_{s'} \to X''_{s'}\]
  indexed by $s' \in S'$, as we expect.
\end{remark}

\begin{remark}\label{rk:span-F-otimes-cocart}
  For later use, let us identify some of the cocartesian edges of the cocartesian fibration $ p\colon\Bispan_{\pb,\teq}(\Ar(\xF))\to\Span(\xF)$. By construction, the restriction of the cartesian fibration $\Span_{\pb}(\Ar(\xF))^{\op}\to\Span(\xF)$ to $\xF$ is the cartesian fibration $\Ar(\xF)\to\xF$, whose cartesian arrows are precisely the pullback squares in $\xF$. Applying the characterization of cocartesian edges from \cite[3.2]{HHLN2}, we therefore see that the $p$-cocartesian lifts of a map
  \begin{equation*}
    S\xleftarrow{\;f\;} T \xrightarrow{\;=\;} T
  \end{equation*}
  in $\Span(\xF)$ are precisely given by the bispans of the form
  \[
    \begin{tikzcd}
      X \arrow{d} & Y \arrow{l} \arrow[r,equals] \arrow{d} \dlpullback & Y \arrow[r,equals] \arrow{d} & Y \arrow{d} \\
      S & T \arrow{l}{f} \arrow[r,equals] & T\arrow[r,equals] & T\rlap.
    \end{tikzcd}
  \]
\end{remark}

\begin{observation}
  Evaluation at $0$ gives a functor
  \[ \Span(\xF)^{\otimes} \simeq
    \Bispan_{\pb,\teq}(\Ar(\xF)) \to
    \Bispan(\xF).\]
  We will prove in Corollary~\ref{cor:spanFotloc} below that this is a localization.
\end{observation}

\subsection{Commutative monoids and Day convolution}\label{sec:dayconv}
In this subsection we briefly recall the definition of commutative
semirings from \cite{GepnerGrothNikolaus} and then study its relation
to the symmetric monoidal structure on $\Span(\xF)$.

\begin{defn}
  We say a symmetric monodial \icat{} $\cC$ is \emph{presentably symmetric
  monoidal} if $\cC$ is presentable and the tensor product
  $\blank\otimes\blank \colon \cC\times\cC \rightarrow \cC$ preserves colimits
  in each variable separately.\footnote{By the adjoint functor theorem for
  presentable \icats{}, this is equivalent to the symmetric monoidal
  structure being closed.} A presentably symmetric monoidal
  $\infty$-category is equivalently a commutative algebra in $\PrL$
  equipped with the Lurie tensor product.
\end{defn}

\begin{thm}[Gepner--Groth--Nikolaus \cite{GepnerGrothNikolaus}*{Theorem~5.1}]\label{thm:ggncmon}
  Suppose $\cC$ is a presentably symmetric monoidal \icat{}. Then
  $\CMon(\cC)$ has a unique presentably symmetric monoidal structure such that the
  free monoid functor $\cC \to \CMon(\cC)$ is
  symmetric monoidal. \qed
\end{thm}

\begin{defn}
  A \emph{commutative semiring} in a presentably symmetric monoidal \icat{} $\cC$
  is a commutative algebra in $\CMon(\cC)$. We write
  \[ \CRig(\cC) := \CAlg(\CMon(\cC))\] for the \icat{} of commutative
  semirings in $\cC$.\footnote{Note that this notation is slightly
    abusive, as this \icat{} really depends not just on $\cC$, but on
    its symmetric monoidal structure.}
\end{defn}

Recall that $\CMon(\cC)$ is by definition a full subcategory of
$\Fun(\Span(\xF), \cC)$. We will now prove that the symmetric monoidal
structure of Theorem~\ref{thm:ggncmon} is a localization of a Day convolution
structure on this functor \icat{}. To see this we first need to know
that this Day convolution indeed localizes to the product-preserving
functors. This is a special case of the following proposition, which
is really just \cite{BenMosheSchlank}*{4.24} stated in a
more general form, but we include a proof for completeness.

\begin{propn}[Ben-Moshe--Schlank]\label{propn:prodprday}
  Fix a set $\cR\subset \Cat_\infty$ of $\infty$-categories. Suppose $\cI$ is an \icat{} with $\cR$-shaped limits that is equipped
  with a symmetric monoidal structure such that for every $X \in \cI$,
  the functor $X \otimes \blank$ preserves $\cR$-shaped limits. Then for every
  presentably symmetric monoidal \icat{} $\cC$, the Day convolution structure on
  $\Fun(\cI, \cC)$ localizes to a symmetric monoidal structure on the
  full subcategory $\Fun^{\cR}(\cI, \cC)$ of $\cR$-limit-preserving
  functors.
\end{propn}
\begin{proof}
  We first consider the case where $\cC$ is the \icat{} $\Spc$ of
  spaces, with the cartesian product.
  Let $L \colon \Fun(\cI, \Spc) \to \Fun^{\cR}(\cI, \Spc)$ denote the
  localization functor. We must show that for any $\Phi \in \Fun(\cI,
  \Spc)$, the functor $\Phi \otimes \blank$ preserves
  $L$-equivalences. Since $L$-equivalences are closed under colimits
  in the arrow \icat{}, it suffices to show this when $\Phi$ is of the
  form $y(X)$ for $X \in \cI^{\op}$, where $y$ is the Yoneda embedding, since these
  generate $\Fun(\cI, \Spc)$ under colimits. Also, for any
  $\Phi$ the statement is equivalent to: for
  $M \in \Fun^{\cR}(\cI, \Spc)$, the internal hom $M^{\Phi}$ in $\Fun(\cI,\Spc)$ also
  preserves $\cR$-shaped limits, as is immediate from the natural equivalence
  \[ \Map(\Phi \otimes \blank, M) \simeq \Map(\blank, M^{\Phi})\]
  and the relation between $L$-equivalences and $L$-local objects.

  Now we claim that we can describe $M^{y(X)}$ as
  $M(\blank \otimes X)$. Indeed, we know that
  $y \colon \cI^{\op} \to \Fun(\cI, \Spc)$ is symmetric monoidal, so
  that we have natural equivalences
  \[
    \begin{split}
      M^{y(X)}(\blank) & \simeq \Map(y(\blank), M^{y(X)}) \\
                       & \simeq \Map(y(\blank) \otimes y(X), M) \\
                       & \simeq \Map(y(\blank \otimes X), M) \\
      & \simeq  M(\blank \otimes X).
    \end{split}
  \]
  Now by assumption $\blank \otimes X \colon \cI \to \cI$ preserves
  $\cR$-shaped limits, hence so does the composite $M(\blank\otimes X)$.

  For a general $\cC$, we have a commutative square
  \[
    \begin{tikzcd}
      \Fun(\cI, \Spc) \otimes \cC \ar[r, "\sim"] \ar[d] & \Fun(\cI, \cC)
      \ar[d] \\
      \Fun^{\cR}(\cI, \Spc) \otimes \cC \ar[r,"\sim"] &
      \Fun^{\cR}(\cI, \cC)
    \end{tikzcd}
  \]
  where the top equivalence is symmetric monoidal for the Day
  convolution by \cite{BenMosheSchlank}*{3.10} and the bottom
  equivalence will be shown below. The result
  then follows from \cite{BenMosheSchlank}*{4.21}, which shows that the left map is again a symmetric monoidal localization.

  It remains to show that the map
  \[
  \Fun^{\cR}(\cI,\Spc)\otimes \cC \rightarrow \Fun^{\cR}(\cI,\cC)
  \]
  is an equivalence. For this we may compute
  \begin{align*}
  \Fun^{\cR}(\cI, \Spc)\otimes \cC &\simeq \Fun^{\mathrm{R}}(\cC^{\op},\Fun^{\cR}(\cI, \Spc))\\
  &\simeq \Fun^{\cR}(\cI,\Fun^{\mathrm{R}}(\cC^{\op},\Spc)) \\
  &\simeq \Fun^{\cR}(\cI,\cC),
  \end{align*}
  where we have used that for any two presentable \icats{} $\cC$ and $\cD$ the tensor product $\cC\otimes \cD$ is equivalent to the \icat{} $ \Fun^{\mathrm{R}}(\cC^{\op},\cD)$ of (small) limit-preserving functors from $\cC^{\op}$ to $\cD$; see \cite{HA}*{4.8.1.17}.
\end{proof}

Applying this to $\Span(\xF)$, we get the following variant of
\cite{BenMosheSchlank}*{4.26}:
\begin{cor}
  Let $\cC$ be a presentably symmetric monoidal \icat{}. The symmetric
  monoidal structure on $\CMon(\cC)$ from \cite{GepnerGrothNikolaus}
  is a localization of the Day convolution on $\Fun(\Span(\xF), \cC)$
  using the symmetric monoidal structure on $\Span(\xF)$ induced by
  the cartesian product in $\xF$.
\end{cor}
\begin{proof}
  The cartesian product on $\xF$ preserves coproducts in each variable separately. As $\xF^\op\hookrightarrow\Span(\xF)$ is essentially surjective and preserves finite products, we conclude that the tensor product on $\Span(\xF)$ (given by the cartesian product in $\xF$) preserves products in
  each variable. Thus the assumptions of Proposition~\ref{propn:prodprday} are
  satisfied, so that
  the Day convolution on $\Fun(\Span(\xF), \cC)$ localizes to
  $\CMon(\cC)$, making this a presentably symmetric monoidal
  \icat{}. It remains to show that this satisfies the condition that
  unqiuely characterizes the symmetric monoidal structure of
  \cite{GepnerGrothNikolaus}, namely that the free commutative monoid
  functor
  \[ F \colon \cC \to \CMon(\cC)\]
  is symmetric monoidal. The functor $F$ is by definition left adjoint
  to the forgetful functor $\CMon(\cC) \to \cC$, which
  factors as
  \[ \CMon(\cC) \hookrightarrow \Fun(\Span(\xF), \cC) \xto{u^{*}}
    \cC,\]
  where $u^{*}$ is evaluation at $\bfone$, \ie{} restriction along $u
  \colon \{\bfone\} \to \Span(\xF)$. Hence $F$ factors as
  \[ \cC \xto{u_{!}} \Fun(\Span(\xF), \cC) \xto{L} \CMon(\cC). \] Here
  the localization $L$ is symmetric monoidal for the localized Day
  convolution, and $u_{!}$ is symmetric monoidal by
  \cite{BenMosheSchlank}*{3.6} since $u \colon * \to \Span(\xF)$ is
  symmetric monoidal.
\end{proof}

Using the universal property of Day convolution, we then immediately
get the following description of commutative semirings:
\begin{cor}\label{cor:crig=lsm}
  For any presentably symmetric monoidal \icat{} $\cC$, the \icat{}
  $\CRig(\cC)$ of commutative semirings in $\cC$ is equivalent to the
  full subcategory of
  $\FunLSM(\Span(\xF)^{\otimes},
  \cC^{\otimes})$ spanned by those lax symmetric monoidal functors
  whose underlying functors $\Span(\xF) \to \cC$ are commutative
  monoids.
\end{cor}
\begin{proof}
  Since $\CMon(\cC)$ is a symmetric monoidal localization of
  $\Fun(\Span(\xF), \cC)$, the inclusion naturally acquires the structure of a
  lax symmetric monoidal functor \cite{GepnerGrothNikolaus}*{3.4}. Postcomposition with this then induces a
  full subcategory inclusion
  \[ \CRig(\cC) = \CAlg(\CMon(\cC)) \subseteq \CAlg(\Fun(\Span(\xF),
    \cC))\] whose image is spanned by the commutative algebras whose
  underlying object in $\Fun(\Span(\xF), \cC)$ is a commutative
  monoid, see e.g.~the proof of~\cite{GepnerGrothNikolaus}*{3.6}.
  Moreover, the universal property of Day convolution from
  \cite{GlasmanDay} (or \cite{HA}*{\S 2.2.6}) identifies
  commutative algebras in $\Fun(\Span(\xF), \cC)$ with
  lax symmetric monoidal functors
  $\Span(\xF)^{\otimes} \to \cC^{\otimes}$.
\end{proof}

In the case where the symmetric monoidal structure on $\cC$ is
cartesian, we can simplify this further:
\begin{cor}\label{cor:crigspanfot}
  For any cartesian closed presentable \icat{} $\cC$, the \icat{}
  $\CRig(\cC)$ (using the cartesian product on $\cC$) is equivalent to the full subcategory of
  $\Fun(\Span(\xF)^{\otimes}, \cC)$ spanned by the functors \[F
  \colon \Span(\xF)^{\otimes} \to \cC\] such that:
  \begin{enumerate}[(1)]
  \item $F$ is a $\Span(\xF)^{\otimes}$-monoid in the sense of Definition~\ref{defn:monoid}.
  \item The restriction of $F$ to $\Span(\xF)$ is a commutative monoid
    in $\cC$.
  \end{enumerate}
\end{cor}
\begin{proof}
  Combine Corollary~\ref{cor:crig=lsm} with the description of lax symmetric
  monoidal functors to $\cC^{\times}$ as monoids from
  Theorem~\ref{thm:lsm=monoid}.
\end{proof}

\section{Comparison with bispans}\label{sec:comparison}

In this section, we first prove a general localization result for
\icats{} of spans in \S\ref{sec:loc}. We then apply this to our
functor $\Span(\xF)^{\otimes} \to \Bispan(\xF)$ in
\S\ref{sec:locbispan} before we complete the proof of our comparison
result for commutative semirings in \S\ref{sec:proof}.

\subsection[Localizations of span $\infty$-categories]{Localizations of span $\bm\infty$-categories}\label{sec:loc}
In this subsection we will prove the following general result about
localizations of \icats{} of spans:
\begin{thm}\label{thm:spanloc}
  Suppose $\phi \colon (\cC, \cC_{F}) \to (\mathcal{D},
  \mathcal{D}_{F})$ is a functor of span pairs such that
  \begin{enumerate}[(i)]
  \item $\phi$ is a localization at some class of maps $W$,
  \item $\cC_{F} \to \mathcal{D}_{F}$ is a right fibration.
  \end{enumerate}
  Then $\Span(\phi) \colon \Span_{F}(\cC) \to \Span_{F}(\mathcal{D})$ is a localization, at the class of maps of the form $X \xfrom{w} Y \xto{=} Y$ where $w$ is in $W$.
\end{thm}

In fact, this theorem is a special case of a more general result about
localizations of factorization systems, in the following sense:

\begin{defn}
  \label{def:Factorization_System}
  A \emph{factorization system} on an $\infty$-category $\cC$ consists
  of two wide subcategories $E,M\subset \cC$ satisfying the following
  conditions:
  \begin{enumerate}
  \item Every morphism in $E$ is \emph{left orthogonal} to every morphism in $M$.
  \item Every morphism $f$ of $\cC$ admits a
    factorization $f \simeq me$ with $e$ in $E$ and $m$
    in $M$.
  \end{enumerate}
  If $(\cC,E,M)$ and $(\cC',E',M')$ are two factorization systems,
  then a \emph{map of factorization systems}
  $f\colon(\cC,E,M)\to (\cC',E',M')$ is a functor
  $f\colon \cC\to \cC'$ of their underlying \icats{} satisfying
  $f(E)\subset E'$ and $f(M)\subset M'$.
\end{defn}

\begin{remark}
  The above definition follows
  \cite{abfj-left-exact}*{3.1.6}. As a consequence of
  {Lemma 3.1.9} of the same paper,
  it is equivalent to
  \cite{HTT}*{5.2.8.8}; see also
  \cite[3.32]{CLL_Clefts}.
\end{remark}

\begin{propn}\label{prop:bc-fact-core}
  Let $(\cC,E,M)$ be a factorization system and consider the
  commutative square of inclusions
  \begin{equation}\label{diag:bc-factorization-system}
    \begin{tikzcd}
      \cC^{\simeq}\arrow[r,hook,"\alpha"]\arrow[d,hook, "i"'] & E\arrow[d,hook, "j"]\\
      M\arrow[r,hook, "\beta"'] & \cC\rlap.
    \end{tikzcd}
  \end{equation}
  Then the induced Beck--Chevalley transformation \[i_!\alpha^*\to\beta^*j_!\]
  on functors to any cocomplete \icat{} is an
  equivalence.
\end{propn}

\begin{proof}
  Using the colimit formula for left Kan extensions, we see that it suffices to show that for every $X\in \cC$, the induced functor
  $\cC^{\simeq}\times_MM_{/X}\to E\times_\cC \cC_{/X}$ is (co)final, \ie{} that restriction along it induces an equivalence on all colimits. We will show that the functor above is a right adjoint, from which the result follows by \cite{CisinskiBook}*{6.1.13}.

  For this, we note that by \cite{CLL_Clefts}*{Proof of 3.33}, the inclusion $\kappa\colon M_{/X}\to \cC_{/X}$
  is fully faithful and admits a left adjoint $\lambda$ (automatically a localization) such that for
  every $(Y\to X)\in \cC_{/X}$ the unit map $Y\to\kappa\lambda(Y)$
  belongs to $E$; moreover, in the same proof it is also shown that
  $\lambda$ inverts all maps in $E$. It follows directly that the
  left adjoint $\lambda\colon\cC_{/X}\to M_{/X}$ restricts to a functor between the non-full
  subcategories
  \[ E\times_\cC \cC_{/X}\to \cC^{\simeq}\times_M M_{/X},\] and that
  also the unit restricts accordingly. As the counit is an equivalence
  by the full faithfulness of $M_{/X}\to \cC_{/X}$, it similarly
  restricts to these subcategories, so that the restricted functor is
  a right adjoint, as claimed.
\end{proof}

Using this, we will now prove the following general base change criterion:
\begin{thm}\label{thm:base-change-general}
	Let
	\begin{equation}\label{eq:factorization-system-BC}
		\begin{tikzcd}
			(\cC_{00},E_{00},M_{00})\arrow[d, "f_0"']\arrow[r, "g_0"] & (\cC_{10}, E_{10}, M_{10})\arrow[d, "f_1"]\\
			(\cC_{01}, E_{01}, M_{01})\arrow[r, "g_1"'] & (\cC_{11}, E_{11}, M_{11})
		\end{tikzcd}
	\end{equation}
	be a commutative diagram of small factorization systems and assume that the restricted maps $f_0\colon M_{00}\to M_{01}$ and $f_1\colon M_{10}\to M_{11}$ are right fibrations.

	Let $\mathcal{T}$ be a complete \icat{} and let $X\colon \cC_{10}\to\mathcal{T}$ be arbitrary. Then the Beck--Chevalley map $g_1^*f_{1*}X\to f_{0*}g_0^*X$ associated to $(\ref{eq:factorization-system-BC})$ is an equivalence if and only if the analogous Beck--Chevalley map $g_1^*f_{1*}(X|_{E_{10}})\to f_{0*}g_0^*(X|_{E_{10}})$ associated to
	\begin{equation*}
		\begin{tikzcd}
			E_{00}\arrow[d, "f_0"']\arrow[r, "g_0"] & E_{10}\arrow[d, "f_1"]\\
			E_{01}\arrow[r, "g_1"'] & E_{11}
		\end{tikzcd}
	\end{equation*}
	is an equivalence.
\end{thm}

We begin with the following key observation (which can also be seen as the special case of the theorem where $\cC_{00}=E_{00}=E_{10}$, $\cC_{01}=E_{01}=E_{11}$, and $g_0,g_1$ are the respective inclusions):

\begin{lemma}\label{lemma:restriction-to-epis}
  Let $f\colon (\cC,E,M)\to (\cC',E',M')$ be a map of small factorization
  systems such that $f\colon M\to M'$ is a right fibration. Then the
  Beck--Chevalley map $j^*f_*\to f_*i^*$ associated to
  \begin{equation*}
    \begin{tikzcd}
      E\arrow[r,hook,"i"]\arrow[d,"f"'] & \cC\arrow[d, "f"]\\
      E'\arrow[r,hook,"j"'] & \cC'
    \end{tikzcd}
  \end{equation*}
  is an equivalence for functors to any complete \icat{}.
\end{lemma}
\begin{proof}
Let $\cT$ be a complete \icat{}, which by changing universe we may assume is small. Then we may embed $\cT$ continuously into a complete and cocomplete \icat{} $\cT'$ via the Yoneda embedding. Proving the statement for $\cT'$ implies the statement for $\cT$, and so it suffices to prove the statement for (co)complete $\cT$. Having made this reduction, we may pass to total mates, and instead show that the
  Beck--Chevalley map $i_!f^*\to f^*j_!$ is an equivalence. For this, let
  us consider the commutative cube
  \begin{equation}\label{diag:inclusions-restr-to-E}
    \begin{tikzcd}[column sep=1.1em, row sep=1.1em]
      & E\arrow[dd, "f"{description, near start,yshift=1pt}]\arrow[rr,hook, "i"] && \cC\arrow[dd, "f"]\\
      \cC^{\simeq}\arrow[dd,"f"']\arrow[ur,hook]\arrow[rr,hook,crossing over,"\alpha"{near end,description}] && M\arrow[ur,hook]\\
      & E'\arrow[rr,hook, "j"{description, near start}] && \cC'\llap.\\
      (\cC')^{\simeq}\arrow[ur,hook] \arrow[rr,hook,"\beta"'] && \arrow[from=uu,crossing over, "f"{near start, description}] M'\arrow[ur,hook]
    \end{tikzcd}
  \end{equation}
  Taking functors into $\mathcal{T}$ and passing to left adjoints horizontally we obtain a commutative diagram
  \begin{equation*}
    \begin{tikzcd}[column sep=1.1em, row sep=1.1em]
      &[-.5em] \mathcal{T}^E\arrow[from=dd, "f^*"{description, near end}]\arrow[rr, "i_!"] &[-.33em]&[-.33em] \mathcal{T}^\cC\arrow[from=dd, "f^*"']\\
      \mathcal{T}^{\mskip.4mu\cC^{\simeq}}\arrow[from=dd,"f^*"]\arrow[from=ur]\arrow[rr,crossing over,"\alpha_!"{near end,description}] && \mathcal{T}^M\arrow[from=ur]\\
      & \mathcal{T}^{E'}\arrow[rr, "j_!"{description, near start}] && \mathcal{T}^{\cC'}\\
      \mathcal{T}^{\mskip.4mu(\cC')^{\simeq}}\arrow[from=ur] \arrow[rr,"\beta_!"'] && \arrow[uu,crossing over, "f^*"{near end, description}] \mathcal{T}^{M'}\arrow[from=ur]
    \end{tikzcd}
  \end{equation*}
  in the homotopy $2$-category of $\infty$-categories, where the left
  and right face are filled by the naturality equivalences and the
  remaining squares are filled by the respective Beck--Chevalley maps,
  which a priori might or might not be invertible; our goal is to
  prove that the natural transformation filling the back square is an
  equivalence.

  For this we observe that the Beck--Chevalley maps in the top and
  bottom face are invertible by
  Proposition~\ref{prop:bc-fact-core}. On the other hand, as
  $f\colon M\to M'$ is conservative, the front face of
  (\ref{diag:inclusions-restr-to-E}) is a pullback square, and as
  this restriction of
  $f$ is moreover assumed to be a right fibration, we see that the
  associated Beck--Chevalley map is an equivalence by the dual of \cite{CisinskiBook}*{6.4.13}.

  Altogether, we have shown that all squares except possibly the back
  square are filled with invertible transformations. By coherence, we
  conclude that $i_!f^*\to f^*j_!$ becomes an equivalence after
  restricting to $M\subset \cC$. However, $M$ is a wide subcategory,
  so restriction is conservative and the claim follows.
\end{proof}

\begin{proof}[Proof of Theorem~\ref{thm:base-change-general}]
  Consider the commutative cube
  \begin{equation*}
    \begin{tikzcd}[column sep=1.1em, row sep=1.1em]
      & \cC_{00}\arrow[dd, "f_0"{description, near start,yshift=1pt}]\arrow[rr, "g_0"] && \cC_{10}\arrow[dd, "f_1"]\\
      E_{00}\arrow[dd,"f_0"]\arrow[ur,hook]\arrow[rr,crossing over,"g_0"{near end,description}] && E_{10}\arrow[ur,hook]\\
      & \cC_{01}\arrow[rr,"g_1"{description, near start}] && \cC_{11}\rlap.\\
      E_{01}\arrow[ur,hook] \arrow[rr,"g_1"'] && \arrow[from=uu,crossing over, "f_1"{near start, description}] E_{11}\arrow[ur,hook]
    \end{tikzcd}
  \end{equation*}
  Mapping into $\mathcal{T}$ and passing to right adjoints vertically, we again obtain a coherent cube, where the top and bottom face are filled by the naturality equivalences and all other faces are filled via the appropriate Beck--Chevalley maps.

  By Lemma~\ref{lemma:restriction-to-epis}, the Beck--Chevalley maps filling the left and right face are invertible. Fixing now $X\colon \cC_{10}\to\mathcal{T}$, we conclude from coherence that the diagram of Beck--Chevalley maps
  \begin{equation*}
    \begin{tikzcd}
      (g_1^*f_{1*}X)|_{E_{01}}\arrow[r]\arrow[d,"\sim"'] & (f_{0*}g_0^*X)|_{E_{01}}\arrow[d,"\sim"]\\
      g_1^*f_{1*}(X|_{E_{10}})\arrow[r] & f_{0*}g_0^*(X|_{E_{10}})
    \end{tikzcd}
  \end{equation*}
  commutes up to homotopy. The claim follows via $2$-out-of-$3$, using again that we can test equivalences on a wide subcategory.
\end{proof}

Using this we can now prove the following localization criterion for maps of factorization systems:

\begin{thm}[`Separation of variables']\label{thm:sep-var}
	Let $f\colon (\cC,E,M)\to (\cC',E',M')$ be a map of factorization systems. Assume the following:
	\begin{enumerate}
		\item $f\colon E\to E'$ is a localization at some class $W\subset E$.
		\item $f\colon M\to M'$ is a right fibration.
	\end{enumerate}
	Then $f\colon \cC\to \cC'$ is  also a localization at $W$.
	\begin{proof}
		Passing to a larger universe, we may assume without loss of generality that all participating \icats{} are small.

		As $f\colon E\to E'$ is a localization, it is in particular essentially surjective, whence so is $f\colon \cC\to \cC'$. It is moreover clear that the latter inverts all maps in $W$. By \cite{CisinskiBook}*{7.1.11} it will therefore be enough to show that for every presentable $\mathcal{T}$ the restriction $f^*\colon\Fun(\cC',\mathcal{T})\to\Fun(\cC,\mathcal{T})$ is fully faithful and that its essential image contains all functors inverting $W$.

		For this we consider the two commutative diagrams
		\begin{equation}\label{diag:degenerate-squares}
			\begin{tikzcd}
				\cC\arrow[r,"f"]\arrow[d, "f"'] & \cC'\arrow[d,equal]\\
				\cC'\arrow[r,equal] & \cC'
			\end{tikzcd}
			\qquad\text{and}\qquad
			\begin{tikzcd}
				\cC\arrow[d,equal]\arrow[r,equal] & \cC\arrow[d, "f"]\\
				\cC\arrow[r, "f"'] & \cC'\rlap.
			\end{tikzcd}
		\end{equation}
		Full faithfulness of $f^*$ amounts to saying that the unit $\id\to f_{*}f^*$ is an equivalence. However, this is precisely the Beck--Chevalley map associated to the left-hand square, so the claim follows from full faithfulness of $(f|_E)^*$ via Theorem~\ref{thm:base-change-general}.

    On the other hand, since $(f|_E)^*$ is fully faithful, a $Y\in\Fun(E,\cT)$ is contained in its essential image (i.e.~inverts $W$) if and only if the counit $(f|_E)^*(f|_E)_*Y\to Y$ is invertible. Applying Theorem~\ref{thm:base-change-general} to the right hand square of (\ref{diag:degenerate-squares}) therefore shows that the counit $f^*f_*X\to X$ is an equivalence if and only if the counit of $X|_{E}$ is so, if and only if $X|_E$ inverts $W$. In particular, any $X\colon\cC\to\cT$ inverting $W$ is contained in the essential image of $f^*$, as claimed.
	\end{proof}
\end{thm}

\begin{proof}[Proof of Theorem~\ref{thm:spanloc}]
Given a span pair $(\cC,\cC_F)$, the \icat{} $\Span_F(\cC)$ admits a
canonical factorization system by \cite{HHLN2}*{4.9} such
that $E = \cC^{\op}$ and $M = \cC_F$. Moreover given any map of span
pairs $\phi\colon
(\cC,\cC_F)\to(\mathcal{D},\mathcal{D}_F)$, the induced map
$\Span(\phi)\colon \Span_F(\cC)\rightarrow \Span_F(\mathcal{D})$ is a map of
factorization systems, which agrees with $\phi^{\op}$ and $\phi_F$ when restricted to
the backward and forward morphisms, respectively. In particular we may immediately
apply the previous criterion.
\end{proof}

\subsection[Localization of $\Span(\xF)^\otimes$]{Localization of $\bm{\textsf{\textbf{Span}}(\xF\kern1pt)^{\otimes}}$}\label{sec:locbispan}
If we want to apply Theorem~\ref{thm:spanloc} to show that
\[
  \Bispan(\ev_0)\colon\Bispan_{\pb,\teq}(\Ar(\xF))\simeq\Span(\xF)^{\otimes}\to\Bispan(\xF)
\]
is a localization, we must first
show that $\Span_{\txt{pb}}(\Ar(\xF)) \to \Span(\xF)$ is a
localization. This is a special case of the following more general statement:
\begin{propn}\label{prop:Span_ev_0_loc}
  Let $(\cC, \cC_{F})$ be a span pair. Then
  $(\Ar(\cC), \Ar(\cC)_{\Fpb})$ is also a span pair,
  where $\Ar(\cC)_{\Fpb}$ is the subcategory whose morphisms
  are the pullback squares
  \[
    \begin{tikzcd}
      X' \arrow{r}{f'} \arrow{d} & Y' \arrow{d} \\
      X \arrow{r}{f} & Y
    \end{tikzcd}
  \]
  where $f$ (and hence also $f'$) lies in $\cC_{F}$. Moreover, $\ev_{0}$ is a morphism of span pairs, and the induced functor
  \[ \Span_{\Fpb}(\Ar(\cC)) \to \Span_{F}(\cC)\]
  is a localization.
\end{propn}
\begin{proof}
	That $(\Ar(\cC), \Ar(\cC)_{\Fpb})$ is again a span pair, and that $\ev_0$ is a map of span pairs, is clear.

	Now consider the constant arrow functor $c \colon \cC \to \Ar(\cC)$, which is left adjoint to $\ev_0$, and note that it is also a map of span pairs. We claim that $\Span(c)$ provides an inverse to $\Span(\ev_0)$ after localizing $\Span_{\Fpb}(\Ar(\cC))$ at the collection maps $W_s$ of maps inverted by $\Span(\ev_0)$. First note that $\Span(\ev_{0}) \circ \Span(c)$ is already the identity before localization. Therefore it suffices to show that the other composite is also the identity after localizing at $W_s$. Consider the counit $\epsilon\colon c\circ \ev_0\Rightarrow \id$ of the adjunction $c\colon\cC\rightleftarrows\Ar(\cC)\noloc\ev_0$. We claim that it induces a natural transformation $\Span(c)\circ \Span(\ev_0)\Rightarrow \id$, which is pointwise given by maps in $W_s$, and hence induces an equivalence after localization. The first claim will follow from \cite{norms}*{C.20}, see also \cite{HHLN2}*{2.22}, after we show that $\eta$ is a natural transformation of span pairs. The second follows by observing that the component of $\Span(\eta)$ on an object $x\to y$ is given by
	\[\begin{tikzcd}
		x & x & x \\
		x & x & y
		\arrow[Rightarrow, no head, from=1-1, to=2-1]
		\arrow[Rightarrow, no head, from=1-2, to=1-3]
		\arrow[Rightarrow, no head, from=1-2, to=1-1]
		\arrow[Rightarrow, no head, from=2-2, to=2-1]
		\arrow[Rightarrow, no head, from=1-2, to=2-2]
		\arrow[from=1-3, to=2-3]
		\arrow[from=2-2, to=2-3]
	\end{tikzcd}\]
	and so is inverted by $\Span(\ev_0)$.

	To complete the proof it only remains to show that $\eta$ is a natural transformation of span pairs. Unwinding the definition, this amounts to the claim that for a commutative cube
	\[
	\begin{tikzcd}[row sep=small, column sep=small]
		x \arrow{rr}  & & x'  \\
		& x  &  & x' \arrow{dd} \\
		x  \arrow{rr}\arrow{dr}  & & x' \arrow{dr} \\
		& y \arrow{rr} \arrow[leftarrow, crossing over]{uu} &  & y'
		\arrow[Rightarrow, no head, from=1-1, to=2-2]
		\arrow[Rightarrow, no head, from=1-1, to=3-1]
		\arrow[Rightarrow, no head, from=1-3, to=2-4]
		\arrow[Rightarrow, no head, from=1-3, to=3-3]
		\arrow[crossing over, from=2-2, to=2-4]
	\end{tikzcd}
	\]
	such that the front face is a pullback in $\cC$, the cube is a
        pullback in $\Ar(\cC)$. Since limits in functor \icats{} are
        computed objectwise, this is true since the back face is
        evidently a pullback.
\end{proof}

Note that the above proposition is not an instance of the localization criterion devised in the previous subsection: the map $\ev_0\colon\Ar(\cC)_{\Fpb}\to\cC_F$ is typically not a right fibration (the target map $\ev_1$ would have been a right fibration, on the other hand). Instead, the next proposition will be Theorem~\ref{thm:spanloc}'s time to shine:

\begin{propn}
  Let $(\cC,\cC_{F}, \cC_{L})$ be a bispan
  triple. Then
  \[(\Ar(\cC), \Ar(\cC)_{\Fpb},
  \Ar(\cC)_{\Lambda})\] is also a bispan triple, where
  $\Ar(\cC)_{\Lambda}$ consists of the squares
  \[
    \begin{tikzcd}
      x \arrow{r}{l} \arrow{d} & y \arrow{d} \\
      x' \arrow{r}{\simeq} & y'
    \end{tikzcd}
  \]
  such that $l$ is in $\cC_{L}$ and the lower horizontal map is an
  equivalence. Moreover, $\ev_{0} \colon \Ar(\cC) \to \cC$ is a morphism of bispan triples, and the induced functor
  \[ \Bispan_{\Fpb,\Lambda}(\Ar(\cC)) \to \Bispan_{F,L}(\cC)\] is a
  localization.
\end{propn}

\begin{proof}
  To show that $(\Ar(\cC), \Ar(\cC)_{\Fpb},\Ar(\cC)_{\Lambda})$ is a
  bispan triple, we first observe that condition (1) in
  Definition~\ref{def:bispantrip} is clear, as is the fact that $\ev_{0}$ gives a morphism of span pairs using both classes of morphisms in $\Ar(\cC)$.
  Now note that given a map
  $f\colon x\rightarrow x'$ in $\cC$, the \icat{}
  $\Ar(\cC)^{\Lambda}_{/f}$ consists of squares of the form
  \[
    \begin{tikzcd}
      y \ar[r, "l"] \ar[d] & x \ar[d, "f"] \\
      y' \ar[r, "\sim"] & x',
    \end{tikzcd}
  \]
  with $l$ in $\cC_{L}$, so that evaluation at $0$ gives an
  equivalence $\Ar(\cC)^{\Lambda}_{/f}\simeq \cC^L_{/x}$, which is moreover
  natural with respect to pullback. Therefore, conditions (2) and (3) in the
  definition follow immediately from the analogous facts for $\cC_F$
  and $\cC_L$. The adjointability condition for $\ev_{0}$ to be a morphism of bispan triples from Definition~\ref{def:bispanmor} also follows immediately.

  Next we recall that, by definition, the functor
  \[ \Bispan_{\Fpb,\Lambda}(\Ar(\cC)) \to \Bispan_{F,L}(\cC)\] is given by applying $\Span$ to the morphism of span pairs
  \[\Span(\ev_0) \colon (\Span_{\Fpb}(\Ar(\cC))^{\op}, \Ar(\cC)_{\Lambda}) \to
    (\Span_{F}(\cC)^{\op}, \cC_L).\] To complete the proof we will
  show that Theorem~\ref{thm:spanloc} applies to this functor. We saw in
  Proposition~\ref{prop:Span_ev_0_loc} that the functor $\Span(\ev_{0})$ is a
  localization, so we only need to show that the functor
  $\ev_0\colon \Ar(\cC)_\Lambda \rightarrow \cC_L$ is a right
  fibration. But here $\Ar(\cC)_\Lambda$ is precisely the subcategory
  of cartesian arrows for the cartesian fibration
  $\ev_{0} \colon \Ar(\cC_L)\rightarrow \cC_L$,
  so we have a right fibration by \cite{HTT}*{2.4.2.5}.
\end{proof}

As a special case, for the bispan triple $(\xF,\xF,\xF)$ we get the
desired localization result for bispans in $\xF$:
\begin{cor}\label{cor:spanFotloc}
  The functor
  \[ \Span(\xF)^{\otimes} \simeq
    \Bispan_{\pb,\teq}(\Ar(\xF)) \to \Bispan(\xF)\]
  given by evaluation at $0$ is a localization. \qed
\end{cor}

\subsection{Proof of the main theorem}\label{sec:proof}
In this subsection we complete the proof of Theorem~\ref{mainthm}. Given Corollary\ref{cor:crigspanfot} and Corollary~\ref{cor:spanFotloc}, it remains to prove the following:
\begin{propn}\label{propn:bispanftreq}
  The following are equivalent for a functor $F\colon \Span(\xF)^{\otimes} \to \cC$:
\begin{enumerate}[(1)]
\item\label{it:c} $F$ is a $\Span(\xF)^\otimes$-monoid, and its restriction to
  $\Span(\xF)$ is a commutative monoid.
\item\label{it:l} $F$ factors through the localization $\Bispan(\xF)$, and the
  induced functor $\Bispan(\xF) \to \cC$ preserves products.
\end{enumerate}
\end{propn}

We first make condition \ref{it:c} above explicit using the equivalence $\Span(\xF)^{\otimes}\simeq \Bispan_{\pb,\teq}(\Ar(\xF))$ from
Corollary~\ref{cor:spanFotimes}:

\begin{lemma}\label{lem:crigcond}
  Let $\cC$ be a presentable \icat{}. For a functor
  \[
    F \colon \Bispan_{\pb,\teq}(\Ar(\xF)) \to \cC,
\]
the following conditions are equivalent:
\begin{enumerate}
\item\label{it:crig} The composite
\[
  \Span(\xF)^\otimes\simeq \Bispan_{\pb,\teq}(\Ar(\xF))\xrightarrow{F}\cC
\]
with the equivalence of Corollary~\ref{cor:spanFotimes} is a $\Span(\xF)^\otimes$-monoid, and its restriction to $\Span(\xF)$ is a commutative monoid.
\item\label{it:prod2} For all {\footnotesize $\smash{\vto{S}{T}}$} in $\Ar(\xF)$, we have equivalences
  \[ F\pvto{S}{T} \isoto \prod_{t \in T} F\pvto{S_{t}}{*}, \quad F\pvto{S}{*} \isoto \prod_{s \in S} F\pvto**,\]
  induced by the backwards maps associated to the obvious inclusions in $\Ar(\xF)$.
\item\label{it:loc} For all {\footnotesize $\smash{\vto{S}{T}}$} in $\Ar(\xF)$, we have equivalences
  \[ F\pvto{S}{*} \isoto F\pvto{S}{T}, \qquad F\pvto{S}{*} \isoto \prod_{s \in S} F\pvto**,\]
    induced by the obvious maps in $\Ar(\xF)$.
\item\label{it:loc2} $F$ takes all backward maps of the form
  \begin{equation}
    \label{eq:bwdgen}
    \begin{tikzcd}
      S \ar[r,equals] \ar[d] & S \ar[d] \\
      T & T' \ar[l]
    \end{tikzcd}
  \end{equation}
  to equivalences, and \[F\pvto{S}{*}\isoto \prod_{s \in S} F\pvto**.\]
\end{enumerate}
\end{lemma}
\begin{proof}
  For the equivalence of \ref{it:crig} and
  \ref{it:prod2} observe that the first condition of \ref{it:prod2} is by Remark~\ref{rk:span-F-otimes-cocart} equivalent to $F$ being a $ \Bispan_{\pb,\teq}(\Ar(\xF))$-monoid, while the usual description of products in $\Span(\xF)$ (as coproducts in $\xF$) shows that the second condition is equivalent to the restriction to $\Span(\xF)$ being a commutative monoid.

  To see that \ref{it:prod2} and \ref{it:loc} are equivalent, consider
  the following commutative square:
  \[
    \begin{tikzcd}
      F\pvto{S}{*} \ar[r, "\text{\ref{it:loc}}"] \ar[d, "\sim"] & F\pvto{S}{T} \ar[r,
      "\text{\ref{it:prod2}}"] & \displaystyle\prod_{t \in T} F\pvto{S_{t}}{*} \ar[d,
      "\sim"] \\
      \displaystyle\prod_{s \in S} F\pvto** \ar[rr, "\sim"] & & \displaystyle\prod_{t \in T}
      \prod_{s \in S_{t}} F\pvto**.
    \end{tikzcd}
  \]
  Here the first top horizontal map is an equivalence if \ref{it:loc} holds
  and the second if \ref{it:prod2} holds,
  while the other maps are invertible under both assumptions; the
  equivalence of the two conditions then follows from the 2-of-3
  property.

  Finally, we observe that \ref{it:loc} is a special case of
  \ref{it:loc2}, and conversely \ref{it:loc2} follows from
  \ref{it:loc} by applying the 2-of-3 property to the value of $F$ at
  the composition
  \[
    \begin{tikzcd}
      S \ar[r,equals] \ar[d] & S \ar[d] \ar[r, equals] & S \ar[d] \\
      * & T \ar[l] & T' \ar[l],
    \end{tikzcd}
  \]
  which completes the proof.
\end{proof}

\begin{proof}[Proof of Proposition~\ref{propn:bispanftreq}]
  Let $S$ be the class of maps in $\Bispan_{\pb,\teq}(\Ar(\xF))$ of
  the form (\ref{eq:bwdgen}) and let $W$ be the class of maps that are
  inverted by $\ev_{0}$. Since a span is invertible \IFF{} both of its
  components are (Observation~\ref{obs:spaneq}), the maps in $W$ are those whose
  top row consists of equivalences, which can immediately be
  simplified to those of the form
  \[
    \begin{tikzcd}
      X \arrow{d} & X \arrow[equals]{l} \arrow[equals]{r} \arrow{d} \drpullback & X \arrow{d} \\
      S & T \arrow{l} \arrow{r} & S'.
    \end{tikzcd}
  \]
  Using the characterization from Lemma~\ref{lem:crigcond}\ref{it:loc2} and
  the description of products in $\Bispan(\xF)$ from \cite{bispans}*{Remark~2.6.13} we see that \ref{it:c} follows from \ref{it:l},
  since $S$ is contained in $W$. For the converse,
  Corollary~\ref{cor:spanFotloc} implies that it suffices to show that a
  functor that inverts the maps in $S$ must invert all maps in $W$. A
  map in $W$ is a composite of a map in $S$ and a forward map of the
  form
  \begin{equation}\label{diag:remains-to-be-inverted}
    \begin{tikzcd}
      X \arrow[equals]{r} \arrow{d} \drpullback & X \arrow{d}{f} \\
      T \arrow{r}[swap]{g} & S,
    \end{tikzcd}
  \end{equation}
  so it suffices to show that such maps are inverted. For this we
  consider the composition
    \[
    \begin{tikzcd}
       X \arrow[equals]{r} \arrow{d} \drpullback & X \arrow{d}{f} & X \arrow[equals]{l} \arrow[equals]{d} \\
       T \arrow{r}[swap]{g} & S & X \arrow{l}{f}
    \end{tikzcd}
  \]
  where the second map lies in $S$. Note that this cospan is a composition of a forward map followed by a backwards map, and by the composition law in $\Span(\Ar(\xF))$ rewriting this as a single span amounts to computing the pullback. The left square being cartesian then implies that the composite is simply
  \[
    \begin{tikzcd}
      X \arrow{d} & X \arrow[equals]{l} \arrow[equals]{r} \arrow[equals]{d} \drpullback & X \arrow[equals]{d} \\
      T & X \arrow{l} \arrow[equals]{r} & X,
    \end{tikzcd}
  \]
  which lies in $S$. The 2-of-3 property then implies that a functor
  that inverts $S$ must invert all maps in $W$, as required.
\end{proof}


\bib{abfj-left-exact}{article}{
  author={Anel, Mathieu},
  author={Biedermann, Georg},
  author={Finster, Eric},
  author={Joyal, Andr\'{e}},
  title={Left-exact localizations of $\infty $-topoi I: Higher sheaves},
  journal={Adv. Math.},
  volume={400},
  date={2022},
  pages={Paper No. 108268, 64},
  eprint={arXiv:2101.02791},
}

\bib{BarwickMackey}{article}{
  author={Barwick, Clark},
  title={Spectral {M}ackey functors and equivariant algebraic $K$-theory ({I})},
  journal={Adv. Math.},
  volume={304},
  date={2017},
  pages={646--727},
  eprint={arXiv:1404.0108},
}

\bib{BarwickMackey2}{article}{
  author={Barwick, Clark},
  author={Glasman, Saul},
  author={Shah, Jay},
  title={Spectral {M}ackey functors and equivariant algebraic $K$-theory, II},
  journal={Tunis. J. Math.},
  volume={2},
  date={2020},
  number={1},
  pages={97--146},
  eprint={arXiv:1505.03098},
}

\bib{BenMosheSchlank}{article}{
  author={Ben-Moshe, Shay},
  author={Schlank, Tomer M.},
  title={Higher semiadditive algebraic K-theory and redshift},
  journal={Compos. Math.},
  volume={160},
  date={2024},
  number={2},
  pages={237--287},
  issn={0010-437X},
  eprint={arXiv:2111.10203},
}

\bib{BermanLawvere}{article}{
  author={Berman, John D.},
  title={Higher Lawvere theories},
  journal={J. Pure Appl. Algebra},
  volume={224},
  date={2020},
  number={9},
  pages={Id/No 106362},
  eprint={arXiv:1903.02991},
}

\bib{norms}{article}{
  author={Bachmann, Tom},
  author={Hoyois, Marc},
  title={Norms in motivic homotopy theory},
  journal={Ast\'{e}risque},
  number={425},
  date={2021},
  eprint={arXiv:1711.03061},
}

\bib{envelopes}{article}{
  author={Barkan, Shaul},
  author={Haugseng, Rune},
  author={Steinebrunner, Jan},
  title={Envelopes for algebraic patterns},
  date={2022},
  eprint={arXiv:2208.07183},
  note={To appear in \textit {Algebr.\ Geom.\ Topol.}},
}

\bib{BoardmanVogt}{book}{
  author={Boardman, J. Michael},
  author={Vogt, Rainer M.},
  title={Homotopy invariant algebraic structures on topological spaces},
  series={Lecture Notes in Mathematics, Vol. 347},
  publisher={Springer-Verlag, Berlin-New York},
  date={1973},
}

\bib{patterns1}{article}{
  author={Chu, Hongyi},
  author={Haugseng, Rune},
  title={Homotopy-coherent algebra via {S}egal conditions},
  date={2021},
  eprint={arXiv:1907.03977},
  journal={Adv. Math.},
  volume={385},
  pages={Paper No. 107733},
}

\bib{CisinskiBook}{book}{
  author={Cisinski, Denis-Charles},
  title={Higher categories and homotopical algebra},
  series={Cambridge Studies in Advanced Mathematics},
  volume={180},
  publisher={Cambridge University Press, Cambridge},
  date={2019},
  pages={xviii+430},
  isbn={978-1-108-47320-0},
  note={Available from \url {https://cisinski.app.uni-regensburg.de/CatLR.pdf}},
}

\bib{CHLL2}{article}{
  title={Normed equivariant ring spectra and higher {Tambara} functors},
  author={Cnossen, Bastiaan},
  author={Haugseng, Rune},
  author={Lenz, Tobias},
  author={Linskens, Sil},
  date={2024},
  eprint={arXiv:2407.08399},
}

\bib{CLL_Clefts}{article}{
  title={Partial parametrized presentability and the universal property of equivariant spectra},
  author={Cnossen, Bastiaan},
  author={Lenz, Tobias},
  author={Linskens, Sil},
  date={2023},
  eprint={arXiv:2307.11001},
}

\bib{CLR}{article}{
  title={Universality of Barwick's Unfurling Construction},
  author={Cnossen, Bastiaan},
  author={Lenz, Tobias},
  author={Ramzi, Maxime},
  date={2025},
  eprint={2502.18278},
}

\bib{cmmn-Mackey}{article}{
  author={Clausen, Dustin},
  author={Mathew, Akhil},
  author={Naumann, Niko},
  author={Noel, Justin},
  title={Descent and vanishing in chromatic algebraic $K$-theory via group actions},
  journal={Ann. Sci. \'Ec. Norm. Sup\'er. (4)},
  volume={57},
  date={2024},
  number={4},
  pages={1135--1190},
  eprint={2011.08233},
}

\bib{CranchThesis}{article}{
  author={Cranch, James},
  title={Algebraic theories and $(\infty ,1)$-categories},
  date={2010},
  eprint={arXiv:1011.3243},
}

\bib{CranchSpan}{article}{
  author={Cranch, James},
  title={Algebraic theories, span diagrams and commutative monoids in homotopy theory},
  date={2011},
  eprint={arXiv:1109.1598},
}

\bib{bispans}{article}{
  author={Elmanto, Elden},
  author={Haugseng, Rune},
  title={On distributivity in higher algebra I: the universal property of bispans},
  journal={Compos. Math.},
  volume={159},
  date={2023},
  number={11},
  pages={2326--2415},
  eprint={arXiv:2010.15722},
}

\bib{GambinoKock}{article}{
  author={Gambino, Nicola},
  author={Kock, Joachim},
  title={Polynomial functors and polynomial monads},
  journal={Math. Proc. Cambridge Phil. Soc.},
  volume={154},
  date={2013},
  pages={153-192},
  eprint={arXiv:0906.4931},
}

\bib{GepnerGrothNikolaus}{article}{
  author={Gepner, David},
  author={Groth, Moritz},
  author={Nikolaus, Thomas},
  title={Universality of multiplicative infinite loop space machines},
  journal={Algebr. Geom. Topol.},
  volume={15},
  date={2015},
  number={6},
  pages={3107--3153},
  eprint={arXiv:1305.4550},
}

\bib{GlasmanDay}{article}{
  author={Glasman, Saul},
  title={Day convolution for $\infty $-categories},
  journal={Math. Res. Lett.},
  volume={23},
  date={2016},
  number={5},
  pages={1369--1385},
  eprint={arXiv:1308.4940},
}

\bib{harpaz}{article}{
  author={Harpaz, Yonatan},
  title={Ambidexterity and the universality of finite spans},
  journal={Proc. Lond. Math. Soc. (3)},
  volume={121},
  date={2020},
  number={5},
  pages={1121--1170},
  eprint={arXiv:1703.09764},
}

\bib{spans}{article}{
  author={Haugseng, Rune},
  title={Iterated spans and classical topological field theories},
  journal={Math. Z.},
  volume={289},
  number={3},
  pages={1427--1488},
  date={2018},
  eprint={arXiv:1409.0837},
}

\bib{anmnd2}{article}{
  author={Haugseng, Rune},
  title={From analytic monads to $\infty $-operads through Lawvere theories},
  date={2023},
  eprint={arXiv:2301.01199},
}

\bib{HHLN2}{article}{
  author={Haugseng, Rune},
  author={Hebestreit, Fabian},
  author={Linskens, Sil},
  author={Nuiten, Joost},
  title={Two-variable fibrations, factorisation systems and $\infty $-categories of spans},
  journal={Forum Math. Sigma},
  volume={11},
  date={2023},
  pages={Paper No. e111},
  eprint={arXiv:2011.11042},
}

\bib{HopkinsLurie}{article}{
  author={Hopkins, Michael J.},
  author={Lurie, Jacob},
  title={Ambidexterity in {$K(n)$}-Local Stable Homotopy Theory},
  note={Available at \url {https://people.math.harvard.edu/~lurie/papers/Ambidexterity.pdf}},
  year={2013},
}

\bib{HTT}{book}{
  author={Lurie, Jacob},
  title={Higher Topos Theory},
  series={Annals of Mathematics Studies},
  publisher={Princeton University Press},
  address={Princeton, NJ},
  date={2009},
  volume={170},
  note={Available from \url {http://math.ias.edu/~lurie/}},
}

\bib{HA}{book}{
  author={Lurie, Jacob},
  title={Higher Algebra},
  date={2017},
  note={Available at \url {http://math.ias.edu/~lurie/}.},
}

\bib{MacphersonCorr}{article}{
  author={Macpherson, Andrew W.},
  title={A bivariant {Y}oneda lemma and $(\infty ,2)$-categories of correspondences},
  date={2020},
  eprint={arXiv:2005.10496},
  journal={Algebr. Geom. Topol.},
  issn={1472-2747},
  volume={22},
  number={6},
  pages={2689--2774},
  year={2022},
}

\bib{MayGeomIter}{book}{
  author={May, J. Peter},
  title={The geometry of iterated loop spaces},
  note={Lecture Notes in Mathematics, Vol. 271},
  publisher={Springer-Verlag, Berlin-New York},
  date={1972},
}

\bib{Stefanich}{article}{
  author={Stefanich, Germ{\'a}n},
  eprint={arXiv:2011.03027},
  title={Higher sheaf theory I: Correspondences},
  date={2020},
}

\bib{StreetPoly}{article}{
  author={Street, Ross},
  title={Polynomials as spans},
  journal={Cah. Topol. G\'{e}om. Diff\'{e}r. Cat\'{e}g.},
  volume={61},
  date={2020},
  number={2},
  pages={113--153},
  issn={1245-530X},
  eprint={arXiv:1903.03890},
}

\bib{StricklandTambara}{article}{
  author={Strickland, Neil},
  title={Tambara functors},
  date={2012},
  eprint={arXiv:1205.2516},
}



\begin{bibdiv}
  \begin{biblist}
  \bibselect{refs_crigbispan}
  \end{biblist}
\end{bibdiv}
\end{document}